\documentclass[12pt]{amsart}

\usepackage{amsmath,amssymb,ifthen}
\usepackage{fullpage}
\usepackage{graphicx,psfrag,subfigure}
\usepackage{color}
\usepackage{moreverb}

\def\R{{\mathbb R}}
\def\N{{\mathbb N}}
\def\EE{{\mathcal E}}
\def\KK{{\mathcal K}}
\def\HH{{\mathcal H}}
\def\OO{{\mathcal O}}
\def\PP{{\mathcal P}}
\def\SS{{\mathcal S}}
\def\TT{{\mathcal T}}
\def\XX{{\mathcal X}}
\def\YY{{\mathcal Y}}
\def\diam{{\rm diam}}

\def\norm#1#2{\|#1\|_{#2}}

\def\enorm#1{|\hspace*{-.5mm}|\hspace*{-.5mm}|#1|\hspace*{-.5mm}|\hspace*{-.5mm}|}

\def\dual#1#2{\langle#1\,,\,#2\rangle}

\def\uu{\boldsymbol u}

\def\vv{\boldsymbol{v}}

\def\xx{\boldsymbol{x}}
\def\yy{\boldsymbol{y}}
\def\ff{\boldsymbol{f}}

\def\LL{\boldsymbol{L}}

\def\div{{\rm div}}
\def\normal{\boldsymbol{n}}

\def\slp{\mathfrak{V}} % simple-layer potential
\def\dlp{\mathfrak{K}} % double-layer potential
\def\hyp{\mathfrak{W}} % hypersingular integral operator
\def\mon{\mathfrak{A}} % monotonte operator

\def\Omegaext{\Omega^{\rm ext}}

\def\strain{\boldsymbol\epsilon}
\def\stress{\boldsymbol\sigma}
\def\H{\boldsymbol{H}}
\def\RR{{\mathcal R}}   % space of rigid body motions
\def\bignull{\boldsymbol{0}} 
\def\linhull{{\rm span}} 
\def\convhull#1{{\rm conv}\left\{ #1 \right\}}
\def\pphi{\boldsymbol \phi}
\def\ppsi{\boldsymbol \psi}
\def\GG{\boldsymbol G}
\def\II{\boldsymbol{I}}
\def\xxi{\boldsymbol \xi}
\def\bb{b}
\def\BB{\mathfrak{B}}
\def\FF{F}

\def\rr{\boldsymbol r}
\def\cchi{\boldsymbol\chi}
\def\ww{\boldsymbol w}
\def\ee{\boldsymbol e}
\def\zz{\boldsymbol z}
\def\aa{\boldsymbol a}
\def\bb{\boldsymbol b}
\def\cc{\boldsymbol c}
\def\dd{\boldsymbol d}
\def\cm{\boldsymbol s}
\renewcommand{\vector}[3]{\begin{pmatrix} #1 \\ #2 \ifthenelse{\equal{#3}{}}{}{\\#3}\end{pmatrix}}
\def\lext{\lambda^{\rm ext}}
\def\mext{\mu^{\rm ext}}
\def\lineg{\mathfrak{g}}
\def\linef{\mathfrak{f}}
\def\lineh{\mathfrak{h}}
\def\stp{\mathfrak{S}}
\def\wwd{\ww^{\rm D}}
\newcounter{constantsnumber}
\def\namec#1#2{%
  \ifthenelse{\equal{#1}{lipschitz}}{C_{\rm lip}}{%
  \ifthenelse{\equal{#1}{monotone}}{C_{\rm mon}}{%
  \ifthenelse{\equal{#1}{cea}}{C_{\mbox{\rm\scriptsize C\'ea}}}{%
  \ifthenelse{\equal{#1}{norm}}{C_{\rm norm}}{%
  \ifthenelse{\equal{#1}{mon}}{{C}_{\rm mon}}{
  \ifthenelse{\equal{#1}{lip}}{{C}_{\rm lip}}{
  \ifthenelse{\equal{#1}{monA}}{c_{\rm mon}}{
  \ifthenelse{\equal{#1}{lipA}}{c_{\rm lip}}{
  \ifthenelse{\equal{#1}{normequiv1}}{c_{\rm norm}}{ 
  \ifthenelse{\equal{#2}{newcounter}}{\refstepcounter{constantsnumber}\label{const#1}}{}C_{\ref{const#1}}}%
  }}}}}}}}}
\def\setc#1{\namec{#1}{newcounter}}
\def\c#1{\namec{#1}{reference}}

\newtheorem{theorem}{Theorem}
\newtheorem{proposition}[theorem]{Proposition}
\newtheorem{lemma}[theorem]{Lemma}

\newenvironment{remark}{\bigskip\noindent\textbf{Remark.}\ \it}{\qed\bigskip}

\def\subsection#1
{
 \bigskip

 \refstepcounter{subsection}
 {\noindent\bf\arabic{section}.\arabic{subsection}.~#1.~}
}

\begin{document}%%%%%%%%%%%%%%%%%%%%%%%%%%%%%%%%%%%%%%%%%%%%%%%%%%%%%

%%%%%%%%%%%%%%%%%%%%%%%%%%%%%%%%%%%%%%%%%%%%%%%%%%%%%%%%%%%%%%%%%%%%%
% Title and Authors
\title[Stability of FEM-BEM couplings in elasticity]{Stability of symmetric and nonsymmetric FEM-BEM couplings for nonlinear
elasticity problems}
\date{\today}

\author{M.~Feischl}
\author{T.~F\"uhrer}
\author{M.~Karkulik}
\author{D.~Praetorius}
\address{Institute for Analysis and Scientific Computing,
       Vienna University of Technology,
       Wiedner Hauptstra\ss{}e 8--10,
       A-1040 Wien, Austria}
\email{\{Michael.Feischl,\,Thomas.Fuehrer,\,Dirk.Praetorius\}@tuwien.ac.at}
\address{Facultad de Matem\'aticas, 
Pontificia Universidad Cat\'olica de Chile, 
Avenida Vicu\~na Mackenna 4860, Santiago, Chile}
\email{mkarkulik@mat.puc.cl}

%%%%%%%%%%%%%%%%%%%%%%%%%%%%%%%%%%%%%%%%%%%%%%%%%%%%%%%%%%%%%%%%%%%%%
% Classification
\keywords{FEM-BEM coupling, elasticity, nonlinearities, well-posedness}
\subjclass[2000]{65N30, 65N15, 65N38}
%%%%%%%%%%%%%%%%%%%%%%%%%%%%%%%%%%%%%%%%%%%%%%%%%%%%%%%%%%%%%%%%%%%%%
% Acknowledgment
\thanks{{\bf Acknowledgment:}
  The research of the authors is supported through the FWF research
  project \textit{Adaptive Boundary Element Method}, see {\tt
  http://www.asc.tuwien.ac.at/abem/}, funded by the Austrian Science Fund (FWF)
  under grant P21732, as well as through the \textit{Innovative Projects
  Initiative} of Vienna University of Technology. This support is thankfully
  acknowledged.
  The authors thank Ernst P. Stephan (University of Hannover) and Heiko
  Gimperlein (University of Copenhagen) for fruitful discussions and careful
  revisions of earlier versions of this manuscript.
  }
%%%%%%%%%%%%%%%%%%%%%%%%%%%%%%%%%%%%%%%%%%%%%%%%%%%%%%%%%%%%%%%%%%%%%
% ABSTRACT
\begin{abstract}
We consider symmetric as well as non-symmetric coupling formulations
of FEM and BEM in the frame of nonlinear elasticity problems. 
In particular, the Johnson-N\'ed\'elec coupling is analyzed.
We prove that these coupling formulations are well-posed and allow for unique
Galerkin solutions if standard discretizations by piecewise polynomials are
employed.
Unlike prior works, our analysis does neither rely on an interior Dirichlet
boundary to tackle the rigid body motions nor on any assumption on the mesh-size
of the discretization used.
\end{abstract}
%%%%%%%%%%%%%%%%%%%%%%%%%%%%%%%%%%%%%%%%%%%%%%%%%%%%%%%%%%%%%%%%%%%%%
% Make Title
\maketitle
%%%%%%%%%%%%%%%%%%%%%%%%%%%%%%%%%%%%%%%%%%%%%%%%%%%%%%%%%%%%%%%%%%%%%
% CONTENTS
%%%%%%%%%%%%%%%%%%%%%%%%%%%%%%%%%%%%%%%%%%%%%%%%%%%%%%%%%%%%%%%%%%%%%%%%%%%%%%%
\section{Introduction \& overview}\label{sec:intro}
%%%%%%%%%%%%%%%%%%%%%%%%%%%%%%%%%%%%%%%%%%%%%%%%%%%%%%%%%%%%%%%%%%%%%%%%%%%%%%%
\noindent
The coupling of the finite element method (FEM) and the boundary element method
(BEM) became very popular when it first appeared in the late seventies of the
last century. These methods combine the advantages of FEM, which allows to
resolve nonlinear problems in bounded domains, and BEM, which allows to solve 
problems with elliptic differential operators with constant coefficients in 
unbounded domains. The two methods are coupled via transmission conditions on 
the coupling boundary.

In 1979,
Zienkiewicz and co-workers~\cite{zkb79} introduced a non-symmetric
one-equation coupling which is based on the first equation of the Calder\'on 
system and only relies on the simple-layer potential $\slp$ as well as the
double-layer potential $\dlp$. In 1980, Johnson \& N\'ed\'elec~\cite{johned} 
gave a first mathematical proof that this coupling procedure is well-posed
and stable. This coupling is therefore also referred to as Johnson-N\'ed\'elec 
coupling. Their analysis relied on Fredholm theory and the compactness of 
$\dlp$ and was thus restricted to smooth coupling boundaries.  
Based on these works,
other coupling methods such as the one-equation Bielak-MacCamy coupling and 
the (quasi-symmetric) Bielak-MacCamy coupling~\cite{bmc} have been proposed.
The requirement for smooth boundaries is a severe restriction when dealing 
with standard FEM or BEM discretizations. Moreover, numerical experiments 
in~\cite{coers} gave empirical evidence that this assumption and hence
the compactness of $\dlp$ can be avoided. It took until 2009 when
Sayas~\cite{sayas09} gave a first mathematical proof for the stability
of the Johnson-N\'ed\'elec coupling on polygonal boundaries. 

In the meantime and because of the lack of satisfying theory, the symmetric 
coupling has been proposed independently by Costabel~\cite{costabel} and 
Han~\cite{han90}. Relying on the symmetric
formulation of the exterior Steklov-Poincar\'e operator,~\cite{costabel,han90}
proved stability of the symmetric coupling. Early works 
including~\cite{costabel,coststeph88,han90,steph92} used interior Dirichlet boundaries 
to tackle constant functions for Laplace 
transmission problems resp.\ rigid body motions for elasticity problems. We 
also refer to the monograph~\cite{gh95} for further details.

To the best of the authors' knowledge, the very first work which avoided the use
of an additional artificial Dirichlet boundary was~\cite{cs1995}, where 
a nonlinear Laplace transmission problem is considered. In the latter work the authors 
used the exterior Steklov-Poincar\'e operator to reduce the coupling equations
to an operator equation with a strongly monotone operator. Although their analysis
avoids an artificial Dirichlet boundary, their proof of ellipticity of the 
discrete exterior Steklov-Poincar\'e operator, and hence of unique solvability
of the discrete coupling equations, involved sufficiently small
mesh-sizes. Bootstrapping the original proof of~\cite{cs1995}, this assumption
could recently be removed~\cite{afp}. 
The authors of~\cite{cfs} then transferred the ideas of~\cite{cs1995} to
nonlinear elasticity problems in 2D.
From an implementational point of view, however, the symmetric coupling seems 
not to be as attractive as the one-equation coupling methods, since all four
integral operators of the Calder\'on system are involved. 

While Sayas' work~\cite{sayas09} focused on the linear Yukawa transmission 
problem as well as the Laplace transmission problem, Steinbach~\cite{s2}
proved stability for a class of linear Laplace transmission problems. He 
introduced an explicit stabilization for the coupling equations so that the
stabilized equations turn out to be elliptic. However, the computation of the stabilization
requires the numerical solution of an additional boundary integral equation at
every discrete level. 
Of \& Steinbach~\cite{os} improved the results from~\cite{s2}, and also gave a
sharp condition under which the stabilized problem is elliptic.
Based on and inspired by the analysis of~\cite{sayas09,s2}, Aurada et 
al.~\cite{affkmp} introduced the idea of \emph{implicit stabilization}. They
proved that all (continuous and discrete) coupling equations are equivalent 
to associated stabilized formulations, even with the same solution. Since 
the stabilized formulations appear to be elliptic, this proves well-posedness 
and stability of the original coupling formulations, i.e.\ no explicit stabilization
is needed or has to be implemented in practice.
For the Johnson-N\'ed\'elec
and Bielak-MacCamy coupling, their analysis covers the same problem class as~\cite{s2}
and moreover extends it to handle certain nonlinearities. For the symmetric coupling,
the analysis of~\cite{affkmp} provides an alternate proof for the results of~\cite{cs1995},
but avoids any restriction on the mesh-size.

In the very recent work~\cite{s3}, Steinbach extended the results
from~\cite{os,s2} to linear elasticity problems. We also refer to~\cite{ghs09}, where 
stability of the Johnson-N\'ed\'elec, the one-equation Bielak-MacCamy, and the 
(quasi-) symmetric Bielak-MacCamy coupling for a Yukawa transmission problem is proven. 
Moreover, they also show that the Johnson-N\'ed\'elec coupling applied to
elasticity problems with interior Dirichlet boundary is stable for certain
specific material parameters.

In our work, we consider (possibly) nonlinear transmission problems in 
elasticity. As a novelty, we introduce a general framework to handle both, the 
symmetric and non-symmetric couplings.  We transfer and extend the idea of 
implicit theoretical stabilization from~\cite{affkmp} to the present setting.
This allows us to prove well-posedness of the non-stabilized coupling 
equations, although they seem to lack ellipticity.
The basic idea is the following: We add appropriate terms to the right-hand side
and left-hand side of the equations and prove that this modified (continuous or 
discrete) problem is equivalent 
to the original problem, even with the same solution. This means that a solution 
of the modified problem also solves the original problem and vice versa. Then, 
we prove existence and uniqueness of the solution of the modified problem and, 
due to equivalence, we infer that the original problem is
well-posed. As in~\cite{cfs,ghs09,s3}, our analysis applies to polygonal 
resp.\ polyhedral coupling boundaries. From our point of view, the advances over 
the state of art are fourfold:
\begin{itemize}
\item Unlike~\cite{cfs}, we do not have to impose any assumption on the mesh-size $h$
in case of the symmetric coupling.
\item Unlike~\cite{coststeph88,gh95,ghs09,han90,steph92}, we avoid the
use of an artificial Dirichlet boundary to tackle the rigid body motions.
\item Unlike~\cite{s3}, we prove well-posedness and stability of the original 
coupling equations and thus avoid any explicit stabilization which requires
the solution of additional boundary integral equations.
\item Unlike~\cite{ghs09,s3}, our analysis for the one-equation couplings also 
covers certain nonlinear material laws, e.g. nonlinear elastic Hencky material
laws.
\end{itemize}

The remainder of this work is organized as follows: In Section~\ref{sec:modelproblem}, 
we state the nonlinear elasticity transmission problem as well as the precise 
assumptions on the nonlinearity.
Furthermore, we fix some notation and collect some important properties of
linear elasticity problems and boundary integral operators, which are used
throughout the work.

Section~\ref{sec:sym} deals with the symmetric coupling. Here, we introduce the
concept of implicit stabilization, and prove unique solvability of the
coupling equations (Theorem~\ref{thm:solvability}). We prove that the
necessary assumption on the BEM discretization is satisfied, if the BEM ansatz
space contains the piecewise constants (Theorem~\ref{thm:assumption}).

In Section~\ref{sec:jn}, we apply the ideas worked out in Section~\ref{sec:sym}
to the Johnson-N\'ed\'elec coupling.
Moreover, we incorporate analytical techniques from~\cite{os,s3} to our method 
and prove unique solvability under an additional assumption on the material 
parameters.

Finally, the short Section~\ref{sec:bmc} analyzes the one-equation Bielak-MacCamy coupling
which seems not to be as present as the symmetric resp. Johnson-N\'ed\'elec
coupling in the literature.

%%%%%%%%%%%%%%%%%%%%%%%%%%%%%%%%%%%%%%%%%%%%%%%%%%%%%%%%%%%%%%%%%%%%%%%%%%%%%%%
\section{Model problem}\label{sec:modelproblem}
%%%%%%%%%%%%%%%%%%%%%%%%%%%%%%%%%%%%%%%%%%%%%%%%%%%%%%%%%%%%%%%%%%%%%%%%%%%%%%%
\noindent
Throughout this work, $\Omega\subseteq\R^d$ ($d=2,3$) denotes a connected
Lipschitz domain with polyhedral boundary $\Gamma=\partial\Omega$ and complement
$\Omegaext = \R^d\backslash\overline\Omega$.

\subsection{Notation}
We use bold symbols for $d$-dimensional vectors, e.g. $\xx$, and vector valued
functions $\uu:\R^d\to\R^d$. The components of such objects will be indexed,
e.g. $\uu=(\uu_1,\uu_2)^T$.
For a set of $n\in\N$ or a sequence of vector-valued objects we use upper
indices for each element of the set resp. sequence, i.e. $\{\uu^j\}_{j=1}^n$
resp. $\{\uu^j\}_{j=1}^\infty$.

Let $X\subseteq\R^d$ be a nonempty, measurable set and let $L^2(X)$ resp.
$H^1(X), H^{1/2}(X) = (H^{-1/2}(X))^*$ denote the usual Lebesgue resp. Sobolev
spaces. We define $\dual{u}{v}_X := \int_X uv\,dx$ for $u,v\in L^2(X)$. For
$u\in H^{-1/2}(\Gamma), v\in H^{1/2}(\Gamma)$, the brackets
$\dual{u}v_\Gamma$ denote the continuously extended $L^2$-scalar product.

For vector-valued Lebesgue resp.\ Sobolev spaces we use bold symbols, i.e. 
$\LL^2(X) := [L^2(X)]^d$ resp. $\H^1(X) := [H^1(X)]^d$ and so on.
Then, we define $\dual\uu\vv_X := \int_X \uu\cdot\vv \,dx$ for
$\uu,\vv\in\LL^2(X)$.
The product space $\HH:= \H^1(\Omega) \times \H^{-1/2}(\Gamma)$, equipped with
the norm $\norm{(\uu,\pphi)}\HH := \big( \norm\uu{\H^1(\Omega)}^2 +
\norm\pphi{\H^{-1/2}(\Gamma)}^2 \big)^{1/2}$ for $(\uu,\pphi)\in\HH$, will be used
throughout the work.
Moreover, let $\strain(\uu) : \stress(\vv) = \sum_{j,k=1}^d
\strain_{jk}(\uu)\stress_{jk}(\vv)$ denote the Frobenius inner product for
arbitrary tensors $\strain,\stress$, and define
$\dual{\stress(\uu)}{\strain(\vv)}_\Omega := \int_\Omega\stress(\uu):\strain(\vv)
\,dx$.
The divergence $\div(\strain(\uu))$ of a tensor is understood component-wise
$(\div(\strain(\uu)))_j = \sum_{k=1}^d \partial \strain_{jk}(\uu) / \partial
x_k$ for $j=1,\ldots,d$.
Finally, we write $\norm{\strain(\uu)}{\LL^2(\Omega)}^2 :=
\dual{\strain(\uu)}{\strain(\uu)}_\Omega$.

\subsection{Linear elasticity}
As usual, the linear and symmetric strain tensor $\strain$ is defined
component-wise by
\begin{align}
  \strain_{jk}(\uu) = \frac12\Big( \frac{\partial\uu_j}{\partial \xx_k} +
  \frac{\partial\uu_k}{\partial \xx_j}\Big)
\end{align}
for all $\uu\in\H^1(\Omega)$ and $j,k=1,\dots,d$.
Together with the Young modulus $E>0$ and the Poisson ratio
$\nu\in(0,\tfrac12)$, the linear stress tensor $\stress$ is defined by
\begin{align}\label{eq:def:stress}
  \stress_{jk}(\uu) = \delta_{jk}\frac{E\nu}{(1+\nu)(1-2\nu)} \div\,\uu + 
  \frac{E}{1+\nu}\strain_{jk}(\uu)
\end{align}
for all $\uu\in\H^1(\Omega)$ and $j,k=1,\dots,d$.
To simplify notation, one usually introduces the so-called Lam\'e constants
\begin{subequations}\label{eq:deflame}
\begin{align}
  \lambda := \frac{E\nu}{(1+\nu)(1-2\nu)} \quad\text{and}\quad 
  \mu := \frac{E}{2(1+\nu)}.
\end{align}
With the identity matrix $\II\in\R^{d\times d}$, the stress tensor $\stress$ then satisfies
\begin{align} 
\begin{split}
  \stress(\uu) &= \lambda\div(\uu)\II + 2\mu\strain(\uu) \quad\text{as well
  as}\\  
  \div\,\stress(\uu) &= \mu\Delta\uu + (\lambda+\mu)\nabla\div(\uu) \quad\text{in
  3D, and} \\
  \div\,\stress(\uu) &= \mu\Delta\uu + \Big(\frac{E\nu}{(1+\nu)(1-\nu)}+\mu\Big)\nabla\div(\uu) \quad\text{in
  2D}.
\end{split}
\end{align}
\end{subequations}
The kernel of the strain tensor $\strain$ is given by the space of rigid body
motions $\RR_d := \ker(\strain) =
\{\vv\in\H^1(\Omega)\,:\,\strain(\vv)=\bignull\}$ which reads 
\begin{align} 
  \RR_2:= \linhull 
             \left\{\begin{pmatrix}1\\0 \end{pmatrix},
                   \begin{pmatrix}0\\1 \end{pmatrix},
                   \begin{pmatrix}-\xx_2\\\xx_1 \end{pmatrix}
		     \right\} \quad\text{for }d=2
\end{align}
and
\begin{align}
  \RR_3 := \linhull 
             \left\{ \begin{pmatrix}1\\0\\0 \end{pmatrix},
                     \begin{pmatrix}0\\1\\0 \end{pmatrix},
                     \begin{pmatrix}0\\0\\1 \end{pmatrix},
                     \begin{pmatrix}-\xx_2\\\xx_1\\0 \end{pmatrix},
                     \begin{pmatrix}0\\-\xx_3\\\xx_2 \end{pmatrix},
                     \begin{pmatrix}\xx_3\\0\\-\xx_1 \end{pmatrix}
		       \right\} \quad\text{for }d=3.
\end{align}
Therefore, it holds $\stress(\vv) = \bignull$ for all
$\vv\in\RR_d$ as well.

\subsection{Nonlinear transmission problem}\label{sec:modelproblem:tp}
As model problem, we consider the following nonlinear transmission problem in free space
\begin{subequations}\label{eq:modelproblem}
\begin{align}
  -\div\,\mon \strain(\uu) &= \ff \quad\text{in } \Omega,
  \label{eq:modelproblem:int} \\
  -\div\,\stress^{\rm ext}(\uu^{\rm ext}) &= \bignull \quad\text{in }\Omegaext,
  \label{eq:modelproblem:ext}\\
  \uu-\uu^{\rm ext} &= \uu_0, \quad\text{on }\Gamma \label{eq:modelproblem:jumpu}\\
  \big(\mon\strain(\uu) - \stress^{\rm ext}(\uu^{\rm ext})\big)\normal &= \pphi_0, \quad\text{on
  }\Gamma, \label{eq:modelproblem:jumpdudn}\\
  |\uu^{\rm ext}(\xx)| &= \OO(1/|\xx|) \quad\text{for }|\xx| \to \infty, \label{eq:modelproblem:radcond}
\end{align}
\end{subequations}
where $\normal$ denotes the exterior unit normal vector on $\Gamma$ pointing from $\Omega$ to
$\Omegaext$. The nonlinear operator $\mon: \R_{\rm sym}^{d\times d} \to \R_{\rm
sym}^{d\times d}$
is used to describe a (possibly) nonlinear material law in $\Omega$. Our
assumptions on the operator $\mon$ and a more detailed description will be
given later on in Section~\ref{sec:modelproblem:mon}.
The stress tensor $\stress^{\rm ext}$, which corresponds to the linear
elasticity problem in the exterior domain, is defined as
in~\eqref{eq:def:stress}--\eqref{eq:deflame} with Lam\'e constants $\lext,\mext$.
For given data $\ff\in\LL^2(\Omega), \uu_0\in\H^{1/2}(\Gamma)$, and $\pphi_0
\in\H^{-1/2}(\Gamma)$, 
problem~\eqref{eq:modelproblem} admits unique solutions $\uu\in\H^1(\Omega)$ and
$\uu^{\rm ext} \in\H_{\rm loc}^1(\Omegaext)$ in 3D. This follows from the
equivalence to the symmetric coupling and its well-posedness, see
Section~\ref{sec:sym}.
For the two-dimensional case, the two-dimensional compatibility condition
\begin{align}\label{eq:compatibility2D}
  \dual{\ff}{\ee^j}_\Omega + \dual{\pphi_0}{\ee^j}_\Gamma = 0 \quad j=1,2
\end{align}
ensures unique solvability. Here, $\ee^j$ are the standard unit normal vectors
in $\R^2$. We refer to~\cite{hw} for further details.

\begin{remark}
  The radiation condition~\eqref{eq:modelproblem:radcond} can be
  generalized to
  \begin{align}\label{eq:genradcond}
    \uu^{\rm ext}(\xx) = -\GG(\xx)\aa + \rr + \OO\Big(|\xx|^{1-d}\Big)
	\quad\text{for } |\xx|\to\infty,
      \end{align}
      with $\rr\in\RR_d$, $\aa\in\R^d$, and $\GG(\cdot)$ being the Kelvin
      tensor defined in~\eqref{eq:defKelvin} below. 
      Moreover, $\aa = \int_\Gamma \stress^{\rm ext}(\uu^{\rm ext})\normal \,d\Gamma$.
      A solution
      of~\eqref{eq:modelproblem:int}--\eqref{eq:modelproblem:jumpdudn}
      with~\eqref{eq:genradcond} is unique. To see this, we stress that the pair
      $(\uu,\uu^{\rm ext})$
      solves~\eqref{eq:modelproblem:int}--\eqref{eq:modelproblem:jumpdudn}
      with~\eqref{eq:genradcond} if and only if the pair $(\widetilde\uu,\widetilde\uu^{\rm
      ext}) = (\uu-\rr,\uu^{\rm ext}-\rr)$
      solves~\eqref{eq:modelproblem:int}--\eqref{eq:modelproblem:jumpdudn} with
      \begin{align}\label{eq:genradcond2}
	\widetilde\uu^{\rm ext}(\xx) = -\GG(\xx)\aa + \OO(1/|\xx|^{1-d})
	\quad\text{for }|\xx|\to\infty
      \end{align}
      and vice versa. Our analysis presented in this work still holds true if we
      replace $(\uu,\uu^{\rm ext})$ by $(\widetilde\uu,\widetilde\uu^{\rm ext})$
      in~\eqref{eq:modelproblem:int}--\eqref{eq:modelproblem:jumpdudn} and the
      radiation condition~\eqref{eq:modelproblem:radcond}
      by~\eqref{eq:genradcond2}.
      Note that $\aa=\bignull$ implies the compatibility
      condition~\eqref{eq:compatibility2D} in 2D. Therefore, the compatibility
      condition can be dropped in 2D for $\aa\neq\bignull$.
      In general, the constant $\aa$ is determined by $\aa = \int_\Omega \ff
      \,d\xx + \int_\Gamma \pphi_0 \,d\Gamma$, which follows
      from~\eqref{eq:modelproblem:int} and~\eqref{eq:modelproblem:jumpdudn}.
      Furthermore, note that $|\GG(\xx)| = \OO(1/|\xx|)$ for $|\xx|\to\infty$
      and $d=3$. Hence,~\eqref{eq:genradcond2} coincides
      with~\eqref{eq:modelproblem:radcond} in 3D.   
\end{remark}

\subsection{Boundary integral operators}\label{sec:intop}
The fundamental solution for linear elastostatics is given by the Kelvin tensor
$\GG(\zz)\in\R_{\rm sym}^{d\times d}$ with
\begin{align}\label{eq:defKelvin} 
  \GG_{jk}(\zz) =
  \frac{\lambda+\mu}{2\mu(\lambda+2\mu)}\left(\frac{\lambda+3\mu}{\lambda+\mu}G(\zz)\delta_{jk}
  + \frac{\zz_j\zz_k}{|\zz|^d}\right)
\end{align}
for all $\zz\in\R^d\backslash\{\bignull\}$ and  $j,k=1,\dots,d$, where $G$ denotes the fundamental solution of the
Laplacian, i.e.
\begin{align}
  G(\zz) =
   \begin{cases}
     -\tfrac1{2\pi} \log |\zz| &\quad\text{for }
     d=2, \\
     \frac1{4\pi} \frac1{|\zz|} &\quad\text{for } d=3.
   \end{cases}
\end{align}
Throughout this work, $\slp$ denotes the simple-layer potential, $\dlp$
the double-layer potential with adjoint $\dlp'$, and $\hyp$ denotes the
hypersingular integral operator.
The boundary integral operators formally read for $\xx\in\Gamma$
as follows:
\begin{align}
  \slp\pphi(\xx) &:= \int_\Gamma \GG(\xx-\yy)\pphi(\yy) \,d\Gamma_{\yy}, \\
  \dlp\vv(\xx) &:= \int_\Gamma \gamma_{1,\yy}^{\rm int}\GG(\xx-\yy)\vv(\yy)  \,d\Gamma_{\yy},  \\
  \hyp\vv(\xx) &:= -\gamma_{1,\xx}^{\rm int} \dlp\vv(\xx),
\end{align}
where $\gamma_{1,\xx}^{\rm int}$ denotes the conormal derivative with respect to
$\xx$ defined in~\eqref{eq:conormal} below.
These operators can be extended to continuous linear operators 
\begin{align}
  \slp &\in L(\H^{-1/2}(\Gamma); \H^{1/2}(\Gamma)), \\
  \dlp &\in L(\H^{1/2}(\Gamma); \H^{1/2}(\Gamma)), \\
  \dlp' &\in L(\H^{-1/2}(\Gamma); \H^{-1/2}(\Gamma)), \\
  \hyp &\in L(\H^{1/2}(\Gamma); \H^{-1/2}(\Gamma)).
\end{align}
We summarize some important properties of these operators. In 3D, the simple-layer
potential is symmetric and elliptic, i.e.\ there holds
\begin{align}
  \dual{\pphi}{\slp\ppsi}_\Gamma = \dual\ppsi{\slp\pphi}_\Gamma
  \quad\text{and}\quad
  \norm{\pphi}{\H^{-1/2}(\Gamma)}^2 \lesssim \dual{\pphi}{\slp\pphi}_\Gamma
  \quad\text{for all }\pphi,\ppsi\in\H^{-1/2}(\Gamma).
\end{align}
Thus, $\norm\pphi\slp := \dual{\pphi}{\slp\pphi}_\Gamma^{1/2}$ defines an equivalent Hilbert norm on
$\H^{-1/2}(\Gamma)$. In 2D, ellipticity can be achieved by an
appropriate scaling of the domain $\Omega$, see e.g.~\cite[Section~6.7]{s} for further
details, and we may thus assume that $\slp$ is elliptic.
The hypersingular operator is symmetric positive semidefinite, i.e. 
\begin{align}\label{eq:hypsemidef}
  \dual{\hyp\vv}\ww_\Gamma =\dual{\hyp\ww}\vv_\Gamma \quad\text{and}\quad
  \dual{\hyp\vv}{\vv}_\Gamma \geq0 \quad\text{for all }\vv,\ww\in\H^{1/2}(\Gamma).
\end{align}
There holds $\ker(\hyp) = \ker(\tfrac12+\dlp) = \RR_d$, see
e.g.~\cite[Section~6.7]{s}.
Throughout this work, the boundary integral operators $\slp,\dlp,\dlp'$, and
$\hyp$ are always understood with respect to the exterior Lam\'e constants
$\lext,\mext$.
We stress that the natural conormal derivative $\gamma_1^{\rm int}$ is
\begin{align}\label{eq:conormal}
  \gamma_1^{\rm int}\uu := \stress(\uu)\normal \quad\text{on }\Gamma.
\end{align}
There holds Betti's first formula, cf. e.g.~\cite[Section~4.2]{s},
\begin{align}\label{eq:betti}
  \dual{\stress(\uu)}{\strain(\vv)}_\Omega = \dual{L\uu}\vv_\Omega +
  \dual{\gamma_1^{\rm int}(\uu)}\vv_\Gamma,
\end{align}
with the linear differential operator $L\uu = -\div\,\stress(\uu)$. 

\subsection{Nonlinear material law and strongly monotone operators}
\label{sec:modelproblem:mon}
We assume $\mon$ to be strongly monotone~\eqref{eq:monA} and Lipschitz
continuous~\eqref{eq:lipA}, i.e. there exist constants $\c{monA}>0$ and $\c{lipA}>0$ such that
\begin{align} \label{eq:monA}
 \c{monA} \norm{\strain(\uu)-\strain(\vv)}{\LL^2(\Omega)}^2 &\leq 
 \dual{\mon\strain(\uu)-\mon\strain(\vv)}{\strain(\uu)-\strain(\vv)}_\Omega,  \quad\text{and }\\
 \label{eq:lipA}
  \norm{\mon\strain(\uu)-\mon\strain(\vv)}{\LL^2(\Omega)} &\leq \c{lipA}
  \norm{\strain(\uu)-\strain(\vv)}{\LL^2(\Omega)}
\end{align}
for all $\uu,\vv\in\H^1(\Omega)$.
In the case $\mon\strain(\cdot) = \stress(\cdot)$ there holds, cf.~\cite[Section~4.2]{s},
\begin{align}\label{eq:stresscont}
  |\dual{\stress(\uu)}{\strain(\vv)}_\Omega| \leq \c{stresscont}
  \norm{\strain(\uu)}{\LL^2(\Omega)}\norm{\strain(\vv)}{\LL^2(\Omega)},
\end{align}
and
\begin{align}\label{eq:stressell}
  \dual{\stress(\uu)}{\strain(\uu)}_\Omega \geq \c{stressell}
  \norm{\strain(\uu)}{\LL^2(\Omega)}^2 
\end{align}
for all $\uu,\vv\in\H^1(\Omega)$, with constants
$\setc{stresscont} = 6\lambda+4\mu$ and $\setc{stressell} = 2\mu$.

An example for a nonlinear material law is the nonlinear elastic Hencky
material, obeying the Hencky-Von Mises stress-strain relation
\begin{align}\label{eq:hencky}
  \mon\strain(\uu) := (K-\tfrac2d \widetilde\mu(\gamma(\strain(u))) ) \div(\uu) \II + 2
  \widetilde\mu(\gamma(\strain(\uu))) \strain(\uu) 
\end{align}
with $K>0$ being the constant bulk modulus and Lam\'e function $\gamma(\strain(\uu)) :=
(\strain(\uu)-\tfrac1d \div(\uu)\II) : (\strain(\uu)-\tfrac1d \div(\uu)\II)$.
Here, $\widetilde\mu : \R_{\geq 0} \to \R_+$ denotes a function such that the operator
from~\eqref{eq:hencky} satisfies~\eqref{eq:monA}--\eqref{eq:lipA}.
Further information on the Hencky material law can be found in 
e.g.~\cite{cfs,coststeph90,steph92,zeidler2} and the references therein.

\subsection{Discretization}
Let $\TT_h$ denote a regular triangulation of $\Omega$ and let $\EE_h^\Gamma$
denote a regular triangulation of $\Gamma$. Here, regularity is understood in
the sense of Ciarlet.
We define the local mesh-width function $h$ by
$h|_X := \diam(X)$ for $X\in\TT_h$ resp. $X\in\EE_h^\Gamma$.  
Moreover, let $\KK_h^\Omega$ denote the set of nodes of $\TT_h$ and let
$\KK_h^\Gamma$ denote the set of nodes of $\EE_h^\Gamma$.
We stress that the triangulation $\EE_h^\Gamma$ of the boundary $\Gamma$ is, in
general, independent of the triangulation $\TT_h$.

Usually, one uses the space $\PP^p(\EE_h^\Gamma) := \{ v\in L^2(\Gamma) \,:\,
v|_E\text{ is a polynomial of degree } \leq p \text{ for all } E\in\EE_h^\Gamma
\}$ to approximate functions $\phi\in H^{-1/2}(\Gamma)$ and the space
$\SS^q(\TT_h) := \PP^q(\TT_h) \cap C(\overline\Omega)$ to approximate functions
$u\in H^1(\Omega)$, with $q=p+1$. Here, $\PP^q(\TT_h) := \{ v\in L^2(\Omega)
\,:\, v|_T \text{ is a polynomial of degree } \leq q\}$.
In Sections~\ref{sec:sym}--\ref{sec:bmc}, we may therefore use the space $\HH_h
:= \XX_h\times\YY_h = \big(\SS^q(\TT_h)\big)^d\times\big(\PP^p(\EE_h^\Gamma)\big)^d$ to approximate
functions $(\uu,\pphi)\in\HH := \H^1(\Omega)\times\H^{-1/2}(\Gamma)$.

%%%%%%%%%%%%%%%%%%%%%%%%%%%%%%%%%%%%%%%%%%%%%%%%%%%%%%%%%%%%%%%%%%%%%%%%%%%%%%%
% SYMMETRIC COUPLING
%%%%%%%%%%%%%%%%%%%%%%%%%%%%%%%%%%%%%%%%%%%%%%%%%%%%%%%%%%%%%%%%%%%%%%%%%%%%%%%
\section{Symmetric FEM-BEM coupling}\label{sec:sym}
\noindent
The symmetric coupling of FEM and BEM has independently been introduced by Costabel
and Han, see~\cite{costabel,han90} for example. It relies on the use of all boundary integral
operators from the Cald\'eron projector.
For the derivation of the variational formulation of the symmetric coupling,
cf.~\eqref{eq:sym:weakform}, we refer to e.g.~\cite{cfs,coststeph90,gh95} for nonlinear
elasticity problems and to e.g.~\cite{affkmp,cs1995,gh95} for nonlinear Laplace
problems.
It is also shown in~\cite{coststeph90} resp. in~\cite{cfs} for the two-dimensional
case that the symmetric coupling~\eqref{eq:sym:weakform} is equivalent to the
model problem~\eqref{eq:modelproblem}.

\subsection{Variational formulation}\label{sec:sym:weakform}
The symmetric coupling reads as follows: Find
$(\uu,\pphi)\in\HH:=\H^1(\Omega)\times\H^{-1/2}(\Gamma)$, such that
\begin{subequations}\label{eq:sym:weakform}
\begin{align}
  \dual{\mon\strain(\uu)}{\strain(\vv)}_\Omega + \dual{\hyp\uu}\vv_\Gamma +
  \dual{(\dlp'-\tfrac12)\pphi}\vv_\Gamma &= \dual\ff\vv_\Omega +\dual{\pphi_0 +
  \hyp\uu_0}\vv_\Gamma, \\
  \dual\ppsi{(\tfrac12-\dlp)\uu + \slp\pphi}_\Gamma &=
  \dual\ppsi{(\tfrac12-\dlp)\uu_0}_\Gamma
  \label{eq2:sym:weakform}
\end{align}
\end{subequations}
holds for all $(\vv,\ppsi)\in\HH$.

To abbreviate notation, we define the mapping $b: \HH\times\HH\to\R$ and the
continuous linear functional $\FF\in\HH^*$ by 
\begin{align}\label{eq:sym:defbb}
\begin{split}
  b( (\uu,\pphi),(\vv,\ppsi)) &:= \dual{\mon\strain(\uu)}{\strain(\vv)}_\Omega \\
  &\qquad+ \dual{\hyp\uu}\vv_\Gamma +
  \dual{(\dlp'-\tfrac12)\pphi}\vv_\Gamma  + 
  \dual\ppsi{(\tfrac12-\dlp)\uu + \slp\pphi}_\Gamma
\end{split}
\end{align}
and
\begin{align}\label{eq:sym:defFF}
  \FF( \vv,\ppsi) := \dual\ff\vv_\Omega +\dual{\pphi_0 +
  \hyp\uu_0}\vv_\Gamma + \dual\ppsi{(\tfrac12-\dlp)\uu_0}_\Gamma
\end{align}
for all $(\uu,\pphi),(\vv,\ppsi)\in\HH$. 
Then, the symmetric coupling~\eqref{eq:sym:weakform} can also be written
as follows: Find $(\uu,\pphi)\in\HH$ such that
\begin{align}\label{eq:sym:weakform2}
  b( (\uu,\pphi),(\vv,\ppsi)) = \FF(\vv,\ppsi) \quad\text{holds for all } (\vv,\ppsi)\in\HH. 
\end{align}
Note that $b(\cdot,\cdot)$ is nonlinear in $\uu$ only, but linear in
$\vv,\ppsi$, and $\pphi$.
If we plug in the functions $(\uu,\pphi)=(\vv,\ppsi) = (\rr,0)$ with
$\rr\in\RR_d$ into~\eqref{eq:sym:defbb}, we observe
\begin{align}\label{eq:sym:notmonotone}
  b( (\rr,0),(\rr,0)) = 0.
\end{align}
Therefore, $b(\cdot,\cdot)$ is not elliptic and unique solvability of
~\eqref{eq:sym:weakform2} cannot be shown directly. In the following sections, we introduce an equivalent
formulation of~\eqref{eq:sym:weakform2} which even has the same solution.
Since this equivalent formulation turns out to be uniquely solvable,
also~\eqref{eq:sym:weakform2} admits a unique solution.

The following two theorems are the main results of this section.
With an additional assumption on the model parameters $\c{monA},\lext,\mext$
these results also hold true for other coupling methods, namely the
Johnson-N\'ed\'elec coupling, cf. Section~\ref{sec:jn}, and the Bielak-MacCamy
coupling, cf. Section~\ref{sec:bmc}.
\begin{theorem}\label{thm:solvability}
Let $\HH_h:=\XX_h\times\YY_h$ be a closed subspace of $\HH$ and assume that
$\YY_0\subseteq\YY_h \cap \LL^2(\Gamma)$ satisfies 
\begin{align}\label{eq:assonYYh}
  \forall \rr\in\RR_d\backslash\{\bignull\}  \exists \xxi\in\YY_0 \quad 
  \dual{\xxi}\rr_\Gamma \neq 0.
\end{align}
Then, the symmetric coupling
\begin{align}\label{eq:sym:weakform4}
  b( (\uu,\pphi), (\vv,\ppsi)) = \FF(\vv,\ppsi)
  \quad\text{for all } (\vv,\ppsi)\in\HH
\end{align}
as well as its Galerkin formulation
\begin{align}\label{eq:sym:weakform:disc}
  b( (\uu_h,\pphi_h), (\vv_h,\ppsi_h)) = \FF(\vv_h,\ppsi_h)
  \quad\text{for all } (\vv_h,\ppsi_h)\in\HH_h
\end{align}
admit unique solutions $(\uu,\pphi)\in\HH$ resp. $(\uu_h,\pphi_h)\in\HH_h$.
Moreover, there holds the C\'ea-type quasi-optimality
\begin{align}\label{eq:sym:cea}
  \norm{(\uu,\pphi)-(\uu_h,\pphi_h)}\HH \leq \c{cea}
  \min_{(\vv_h,\ppsi_h)\in\HH_h} \norm{(\uu,\pphi)-(\vv_h,\ppsi_h)}\HH.
\end{align}
The constant $\c{cea}>0$ depends only on $\Omega$, $\mon$, $\YY_0$, and on the
Lam\'e constants $\lext,\mext$.
\end{theorem}

Assumption~\eqref{eq:assonYYh} is clearly satisfied if
$\YY_0:=\big(\PP^1(\EE_h^\Gamma)\big)^d$
denotes the space of affine functions restricted to $\EE_h^\Gamma$,
since $\RR_d\subseteq\big(\PP^1(\EE_h^\Gamma)\big)^d$ and one may thus choose $\xxi=\rr$ in~\eqref{eq:assonYYh}.
However, we shall also show that the space
$\YY_0 := \big(\PP^0(\EE_h^\Gamma)\big)^d$ is sufficiently rich to
ensure~\eqref{eq:assonYYh}. This is precisely the second theorem
we aim to emphasize and prove.
Note that the constant $\c{cea}$ does not depend on the mesh-size $h$ if $\YY_0
\subseteq \YY_h$ for all $h$.

\begin{theorem}\label{thm:assumption}
For $d=2,3$, the space $\YY_0 := \big(\PP^0(\EE_h^\Gamma)\big)^d$ satisfies 
assumption~\eqref{eq:assonYYh}.
\end{theorem}
\noindent
The proof of Theorem~\ref{thm:solvability} resp.\ Theorem~\ref{thm:assumption}
is carried out in Section~\ref{sec:proof:solvability} resp.\
Section~\ref{sec:proof:assumption}.

%---------------------------------------------------------------------------
\subsection{Implicit theoretical stabilization}\label{sec:sym:stab}
%---------------------------------------------------------------------------
To prove Theorem~\ref{thm:solvability}, we shall add appropriate terms to the
linear form $b(\cdot,\cdot)$, which tackle the rigid body motions in the interior domain
$\Omega$.
These (purely theoretical) stabilization terms are chosen in such a way that they vanish when
inserting a (continuous resp.\ discrete) solution of~\eqref{eq:sym:weakform4}.
To be more precise, we will use~\eqref{eq2:sym:weakform} to
stabilize the linear form $b(\cdot,\cdot)$. 

\begin{proposition}\label{prop:equivalence}
Let $\HH_h = \XX_h\times\YY_h$ be a closed subspace of $\HH$.
Let $\{(\xxi^j)_{j=1}^D\}\subseteq\YY_h$, $D\in\N$, be a set of linearly independent functions.
  Define
  \begin{align}
    \widetilde b( (\uu,\pphi),(\vv,\ppsi) ) :=  b( (\uu,\pphi),(\vv,\ppsi) ) + 
    \sum_{j=1}^D \dual{\xxi^j}{(\tfrac12\!-\!\dlp)\uu \!+\! \slp\pphi}_\Gamma
    \dual{\xxi^j}{(\tfrac12\!-\!\dlp)\vv \!+\! \slp\ppsi}_\Gamma
  \end{align}
  for all $(\uu,\pphi),(\vv,\ppsi)\in\HH_h$ and
  \begin{align} 
    \widetilde\FF(\vv,\ppsi) := \FF(\vv,\ppsi) + \sum_{j=1}^D
    \dual{\xxi^j}{(\tfrac12-\dlp)\uu_0}
    \dual{\xxi^j}{(\tfrac12-\dlp)\vv+\slp\ppsi}_\Gamma
  \end{align}
  for all $(\vv,\ppsi)\in\HH_h$.
  Then, there holds the following equivalence: A function $(\uu,\pphi)\in\HH_h$ solves
  \begin{align}\label{eq:sym:weakform3}
    b( (\uu,\pphi),(\vv,\ppsi) ) = \FF(\vv,\ppsi) \quad\text{for all } (\vv,\ppsi)\in\HH_h 
  \end{align}
  if and only if it also solves
  \begin{align}\label{eq:sym:weakformmod}
    \widetilde b( (\uu,\pphi),(\vv,\ppsi) ) = \widetilde\FF(\vv,\ppsi)
    \quad\text{for all }(\vv,\ppsi)\in\HH_h.
  \end{align}
\end{proposition}

\begin{proof}
  \textbf{Step~1. }  Assume that $(\uu,\pphi)\in\HH_h$
  solves~\eqref{eq:sym:weakform3}. Firstly, by inserting the test-function
  $(\bignull,\xxi^j)$ in~\eqref{eq:sym:weakform3}, we get
  \begin{align*} 
    \dual{\xxi^j}{(\tfrac12-\dlp)\uu + \slp\pphi}_\Gamma 
    =b( (\uu,\pphi),(\bignull,\xxi^j)) = \FF(\bignull,\xxi^j) =
    \dual{\xxi^j}{(\tfrac12-\dlp)\uu_0}_\Gamma    
  \end{align*}
  for all $j=1,\dots,D$. 
  Secondly, we multiply this equation with
  $\dual{\xxi^j}{(\tfrac12-\dlp)\vv+\slp\ppsi}_\Gamma$ and infer
  \begin{align*}
    \dual{\xxi^j}{(\tfrac12-\dlp)\uu +\slp\pphi}_\Gamma 
    \dual{\xxi^j}{(\tfrac12-\dlp)\vv+\slp\ppsi}_\Gamma 
    =  \dual{\xxi^j}{(\tfrac12-\dlp)\uu_0}_\Gamma 
     \dual{\xxi^j}{(\tfrac12-\dlp)\vv+\slp\ppsi}_\Gamma 
   \end{align*}
  Last, we sum up these terms over all $j=1,\dots,D$
  and add the sum to~\eqref{eq:sym:weakform3} to see that
  $(\uu,\pphi)$ solves~\eqref{eq:sym:weakformmod}.
  \\
  \textbf{Step~2. } Assume that $(\uu,\pphi)\in\HH_h$
  solves~\eqref{eq:sym:weakformmod}. By choosing $(\vv,\ppsi) =
  (\bignull,\xxi^\ell)$ as a test-function in~\eqref{eq:sym:weakformmod}, we infer
  \begin{align*} 
    &\dual{\xxi^\ell}{(\tfrac12-\dlp)\uu+\slp\pphi}_\Gamma +\sum_{j=1}^D
    \dual{\xxi^j}{(\tfrac12-\dlp)\uu+\slp\pphi}_\Gamma
    \dual{\xxi^j}{\slp\xxi^\ell}_\Gamma
     = \widetilde b( (\uu,\pphi),(\bignull,\xxi^\ell))
    \\&\qquad
    = \widetilde \FF(\bignull,\xxi^\ell) 
     = \dual{\xxi^\ell}{(\tfrac12-\dlp)\uu_0}_\Gamma +\sum_{j=1}^D
    \dual{\xxi^j}{(\tfrac12-\dlp)\uu_0}_\Gamma
    \dual{\xxi^j}{\slp\xxi^\ell}_\Gamma,
  \end{align*}
  which is equivalent to
  \begin{align}\label{eq:sym:modproof} 
    \sum_{j=1}^D \dual{\xxi^j}{(\tfrac12-\dlp)(\uu-\uu_0)+\slp\pphi}_\Gamma
    \dual{\xxi^j}{\slp\xxi^\ell}_\Gamma = -
    \dual{\xxi^\ell}{(\tfrac12-\dlp)(\uu-\uu_0)+\slp\pphi}_\Gamma
  \end{align}
  for all $\ell=1,\dots,D$.
  Next, we define a matrix $\boldsymbol A\in\R_{\rm sym}^{D\times D}$ with entries
  ${\boldsymbol A}_{jk} :=
  \dual{\xxi^k}{\slp\xxi^j}_\Gamma$ and a vector $\xx\in\R^D$ with entries
  $\xx_k:=\dual{\xxi^k}{(\tfrac12-\dlp)(\uu-\uu_0)+\slp\pphi}_\Gamma$ for
  all $j,k=1,\dots,d$.
  Then, we can rewrite~\eqref{eq:sym:modproof} for all $\ell=1,\dots,D$ simultaneously as 
  \begin{align}\label{eq:sym:modproof2}
    {\boldsymbol A}\xx = -\xx.
  \end{align}
  Since $\slp$ is elliptic, and $(\xxi^j)_{j=1}^D$ are linearly independent, the
  matrix $\boldsymbol A$ is
  positive definite and thus only has positive eigenvalues.
  Therefore,~\eqref{eq:sym:modproof2} is equivalent to $\xx=\bignull$,
  which means 
  \begin{align*}
    \dual{\xxi^j}{(\tfrac12-\dlp)\uu+\slp\pphi}_\Gamma =
    \dual{\xxi^j}{(\tfrac12-\dlp)\uu_0}_\Gamma
  \end{align*}
  for all $j=1,\dots,D$.
  With these equalities and the definitions of $b(\cdot,\cdot)$ and $\widetilde
  b(\cdot,\cdot)$, we get
  \begin{align*}
    \widetilde b( (\uu,\pphi),(\vv,\ppsi)) - b( (\uu,\pphi),(\vv,\ppsi)) &=
    \sum_{j=1}^D \dual{\xxi^j}{(\tfrac12-\dlp)\uu+\slp\pphi}_\Gamma
    \dual{\xxi^j}{(\tfrac12-\dlp)\vv+\slp\ppsi}_\Gamma \\
    &= \sum_{j=1}^D \dual{\xxi^j}{(\tfrac12-\dlp)\uu_0}_\Gamma
    \dual{\xxi^j}{(\tfrac12-\dlp)\vv+\slp\ppsi}_\Gamma \\
    &= \widetilde\FF(\vv,\ppsi) - \FF(\vv,\ppsi).
  \end{align*}
  In particular,~\eqref{eq:sym:weakformmod} thus
  implies~\eqref{eq:sym:weakform3}. This concludes the proof.
\end{proof}

%------------------------------------------------------------------------------
\subsection{Equivalent norm}\label{sec:sym:norm}
%------------------------------------------------------------------------------
To show that the equivalent bilinear form of Proposition~\ref{prop:equivalence} 
is, in fact, stabilized and yields a strongly elliptic formulation, we will
prove that the employed stabilization term provides an
equivalent norm on the energy space $\HH$.

\begin{lemma}\label{lemma:equivnorm}
  Let $g_j : \HH\to\R$ with $j=1,\dots,D$ denote linear and continuous functionals such that
  \begin{align}\label{eq:equivnorm1}
    |g(\rr,\bignull)|^2 := \sum_{j=1}^D g_j(\rr,\bignull)^2 \neq 0
    \quad\text{holds for all }\rr\in\RR_d\backslash\{\bignull\}.
  \end{align}
  Then, the definition
   \begin{align}
    \enorm{(\uu,\pphi)}^2 := \norm{\strain(\uu)}{L^2(\Omega)}^2 +
    \dual{\pphi}{\slp\pphi}_\Gamma + |g(\uu,\pphi)|^2
   \quad\text{for all }(\uu,\pphi)\in\HH
  \end{align}
  yields an equivalent norm on $\HH$, and the
  norm equivalence constant $\setc{norm}>0$ in
  \begin{align}
   \c{norm}^{-1}\, \norm{(\uu,\pphi)}{\HH}
   \le \enorm{(\uu,\pphi)}
   \le \c{norm}\, \norm{(\uu,\pphi)}{\HH}
   \quad\text{for all }(\uu,\pphi)\in\HH
  \end{align}
   depends only on $\Omega$, $\lext$, $\mext$, and $g$.
\end{lemma}

\begin{proof}
  Firstly, due to boundedness of $g$ and $\norm{\strain(\uu)}{\LL^2(\Omega)}
  \lesssim \norm\uu{\H^1(\Omega)}$ there holds $\enorm{(\uu,\pphi)}\lesssim
  \norm{(\uu,\pphi)}\HH.$
  Secondly, we argue by contradiction to prove the converse estimate:
  Assume that there exist functions $(\uu_n,\pphi_n)$ with $\norm{(\uu_n,\pphi_n)}\HH > n
  \enorm{(\uu_n,\pphi_n)}$ for all $n\in\N$. Define
  \begin{align*}
    (\vv_n,\ppsi_n) := \frac{(\uu_n,\pphi_n)}{\norm{(\uu_n,\pphi_n)}\HH}.
  \end{align*}
  Then, it follows $\enorm{(\vv_n,\ppsi_n)} < 1/n$ and thus
  $\strain_{jk}(\vv_n)\to 0$ in $L^2(\Omega)$ for $j,k=1,\dots,d$ as well as $\ppsi_n\to\bignull$ in
  $\H^{-1/2}(\Gamma)$.
  By definition of $(\vv_n,\ppsi_n)$, there holds $\norm{(\vv_n,\ppsi_n)}\HH =
  1$ and we may extract a weakly convergent subsequence with
  $(\vv_{n_\ell},\ppsi_{n_\ell}) \rightharpoonup (\vv,\ppsi)$ in $\HH$.
  Next, we conclude that $\ppsi_{n_\ell}\to\ppsi=\bignull$ in
  $\H^{-1/2}(\Gamma)$ and $\vv_{n_\ell}\to\vv$ in $\LL^2(\Omega)$, where the
  latter follows from weak convergence $\vv_{n_\ell}\rightharpoonup\vv$ in
  $\H^1(\Omega)$ and the Rellich compactness theorem.
  Weak lower semi-continuity of $\enorm{(\cdot,\cdot)}$ implies $\enorm{(\vv,\ppsi)}=0$ and thus
  $\strain(\vv)=\bignull$ and $|g( \vv,\bignull)|=0$. Due to
  $\ker(\strain) =\RR_d$ and~\eqref{eq:equivnorm1}, it follows that
  $\vv=\bignull$. 
  Moreover, with \textit{Korn's second inequality}, cf. e.g.~\cite[Theorem~4.17]{s},
  we infer
  \begin{align*}
    \norm{\vv_{n_\ell}-\vv}{\H^1(\Omega)}^2 \lesssim
    \norm{\strain(\vv_{n_\ell})-\strain(\vv)}{\LL^2(\Omega)}^2 +
    \norm{\vv_{n_\ell}-\vv}{\LL^2(\Omega)}^2 \xrightarrow{\ell\to\infty} 0
  \end{align*}
  and therefore
  $(\vv_{n_\ell},\ppsi_{n_\ell})\to(\bignull,\bignull)$ in $\HH$, which contradicts
  $\norm{(\vv_{n_\ell},\ppsi_{n_\ell})}\HH = 1$. This concludes the proof.
\end{proof}

The following proposition provides the equivalent norm used to analyze the
symmetric coupling as well as the Johnson-N\'ed\'elec coupling (see
Section~\ref{sec:jn} below).

\begin{proposition}\label{prop:richenough}
  Let $\YY_0\subseteq\YY\cap\LL^2(\Gamma)$ be a subspace which satisfies 
  assumption~\eqref{eq:assonYYh} of Theorem~\ref{thm:solvability}.
  Let $\rr^1,\dots,\rr^D$ with $D=\dim(\RR_d)$ denote a basis of the rigid body
  motions and let $\Pi_0 : \LL^2(\Gamma) \to \YY_0$ be the $\LL^2$-orthogonal
  projection. Then, $\xxi^j:=\Pi_0(\rr^j)$ for $j=1,\dots,D$ are linearly
  independent. Moreover, the functionals $g_j \in\HH^*$ defined by
  \begin{align}\label{eq:richenough:gj}
    g_j(\uu,\pphi) := \dual{\xxi^j}{(\tfrac12-\dlp)\uu+\slp\pphi}_\Gamma
    \quad\text{for }(\uu,\pphi)\in\HH
  \end{align}
  fulfill assumption~\eqref{eq:equivnorm1} of Lemma~\ref{lemma:equivnorm}.
  In particular,
  \begin{align}\label{eq:richenough:g}
    \enorm{(\uu,\pphi)}^2 := \norm{\strain(\uu)}{\LL^2(\Omega)}^2 +
    \dual\pphi{\slp\pphi}_\Gamma + \sum_{j=1}^D
    |\dual{\xxi^j}{(\tfrac12-\dlp)\uu+\slp\pphi}_\Gamma|^2
  \end{align}
  is an equivalent norm on $\HH$, and the norm equivalence constant $\setc{norm}>0$ in
  \begin{align}
   \c{norm}^{-1}\, \norm{(\uu,\pphi)}{\HH}
   \le \enorm{(\uu,\pphi)}
   \le \c{norm}\, \norm{(\uu,\pphi)}{\HH}
   \quad\text{for all }(\uu,\pphi)\in\HH
  \end{align}
  depends only on $\Omega$, $\YY_0$, $\lext$, and $\mext$.
\end{proposition}

\begin{proof}
  We work out an alternative formulation of~\eqref{eq:assonYYh}.
  With $\dual\xxi\rr_\Gamma = \dual{\Pi_0 \xxi}\rr_\Gamma = \dual\xxi{\Pi_0
  \rr}_\Gamma$, condition~\eqref{eq:assonYYh} becomes
  \begin{align*}
    \forall\rr\in\RR_d\backslash\{\bignull\}\exists\xxi\in\YY_0
    \quad\dual{\xxi}{\Pi_0\rr}_\Gamma \neq 0.
  \end{align*}
  Clearly, this is equivalent to $\Pi_0(\rr)\neq0$ for all $\rr\in\RR_d
  \backslash\{\bignull\}$, which yields that the functions
    $\xxi^j:=\Pi_0(\rr^j)$, for $j=1,\dots,D$, {are linearly independent.}
  Therefore, we can reformulate condition~\eqref{eq:assonYYh} as
  \begin{align}\label{eq:assonYYh2}
    \forall\rr\in\RR_d\backslash\{\bignull\}\exists j\in\{1,\dots,D\} \quad
    \dual{\xxi^j}{\rr}_\Gamma \neq 0.
  \end{align} 
  The functionals $g_j$ are well-defined, linear, and bounded. To
  see~\eqref{eq:equivnorm1}, we stress that due to $\ker(\tfrac12+\dlp)= \RR_d$,
  \begin{align*}
    g_j(\rr,\bignull) = \dual{\xxi^j}{(\tfrac12-\dlp)\rr}_\Gamma =
    \dual{\xxi^j}\rr_\Gamma \quad\text{for }j=1,\dots,D \text{ and }
    \rr\in\RR_d.
  \end{align*}
  From~\eqref{eq:assonYYh2} we infer that there exists $j\in\{1,\dots,D\}$ such
  that $g_j(\rr,\bignull)\neq0$. Therefore,~\eqref{eq:equivnorm1} holds for 
  \begin{align*}
    |g(\uu,\pphi)|^2 = \sum_{j=1}^D g_j(\uu,\pphi)^2 = \sum_{j=1}^D
    |\dual{\xxi^j}{(\tfrac12-\dlp)\uu+\slp\pphi}_\Gamma|^2.
  \end{align*}
  This concludes the proof.
\end{proof}

%-----------------------------------------------------------------------------------------------------------
\subsection{Proof of Theorem~\ref{thm:solvability}}\label{sec:proof:solvability}
%-----------------------------------------------------------------------------------------------------------
As far as existence and uniqueness of solutions is concerned, it suffices to
consider the Galerkin formulation~\eqref{eq:sym:weakform:disc}, since this covers 
the case $\HH_h=\HH$ as well. With assumption~\eqref{eq:assonYYh} and
$\YY_0\subseteq\YY_h \cap \LL^2(\Gamma)$, Proposition~\ref{prop:richenough} allows to
apply Proposition~\ref{prop:equivalence}. Hence, we may equivalently ask for
the unique solvability of~\eqref{eq:sym:weakformmod} instead
of~\eqref{eq:sym:weakform3} resp.~\eqref{eq:sym:weakform:disc}.
To this end, we define the nonlinear operator $\widetilde\BB : \HH_h \to \HH_h^*$ by 
  \begin{align*}
    \widetilde\BB(\uu_h,\pphi_h) := \widetilde b( (\uu_h,\pphi_h),\cdot).
  \end{align*} 
  First, we rewrite
  equation~\eqref{eq:sym:weakformmod} as an equivalent operator equation:
  Find
  $(\uu_h,\pphi_h)\in\HH_h$ such that
  \begin{align}\label{eq:sym:operatoreq}
    \widetilde\BB(\uu_h,\pphi_h) = \FF \quad\text{in }\HH_h^*.
  \end{align}
  
  \textbf{Step~1 }(\textit{Lipschitz continuity of $\widetilde\BB$}).
  Due to the Lipschitz continuity~\eqref{eq:lipA} of $\mon$ and the boundedness
  of the boundary integral operators, it clearly follows that $\widetilde\BB$ is
  also Lipschitz continuous.
  The Lipschitz constant $\c{lip}>0$ in
  \begin{align}\label{eq:sym:lipBB}
   \norm{\widetilde\BB(\uu_h,\pphi_h)-\widetilde\BB(\vv_h,\ppsi_h)}{\HH^*} \leq \c{lip}
   \norm{(\uu_h,\pphi_h)-(\vv_h,\ppsi_h)}\HH,
  \end{align}
  for all $(\uu_h,\pphi_h),(\vv_h,\ppsi_h)\in\HH$,
  thus depends  only on $\mon,\Omega, \lext$, and $\mext$.
  
  \textbf{Step~2 }(\textit{Strong monotonicity of~$\widetilde\BB$}).
  We have to prove that, for all $(\uu_h,\pphi_h),(\vv_h,\ppsi_h)\in\HH$,
  \begin{align}
    \dual{\widetilde\BB(\uu_h,\pphi_h)-\widetilde\BB(\vv_h,\ppsi_h)}{(\uu_h-\vv_h,\pphi_h-\ppsi_h)}
    \geq \c{mon} \norm{(\uu_h-\vv_h,\pphi_h-\ppsi_h)}{\HH}^2.
  \end{align}
  To abbreviate notation, let $(\ww_h,\cchi_h) :=
  (\uu_h-\vv_h,\pphi_h-\ppsi_h)$. Then, we get
  \begin{align*}
    &\dual{\widetilde\BB(\uu_h,\pphi_h)-\widetilde\BB(\vv_h,\ppsi_h)}{(\ww_h,\cchi_h)}
    \\
    &\quad= \dual{\mon\strain(\uu_h)-\mon\strain(\vv_h)}{\strain(\ww_h)}_\Omega  
    + \dual{\hyp\ww_h}{\ww_h}_\Gamma +
    \dual{(\dlp'-\tfrac12)\cchi_h}{\ww_h}_\Gamma \\ 
    &\qquad\qquad +\dual{\cchi_h}{(\tfrac12-\dlp)\ww_h+\slp\cchi_h}_\Gamma + \sum_{j=1}^D
    |\dual{\xxi^j}{(\tfrac12-\dlp)\ww_h+\slp\cchi_h}_\Gamma|^2 =: I
  \end{align*}
  Next, we use strong monotonicity~\eqref{eq:monA} of $\mon$ and positive
  semi-definiteness~\eqref{eq:hypsemidef} of
  $\hyp$ to estimate
  \begin{align*}
    I &\geq \c{monA} \norm{\strain(\ww_h)}{\LL^2(\Omega)}^2 +
    \dual{\cchi_h}{\slp\cchi_h}_\Gamma + \sum_{j=1}^D
    |\dual{\xxi^j}{(\tfrac12-\dlp)\ww_h+\slp\cchi_h}_\Gamma|^2 \\
    & \geq  \min\{\c{monA},1\} \, \Big(\norm{\strain(\ww_h)}{\LL^2(\Omega)}^2 +
    \dual{\cchi_h}{\slp\cchi_h}_\Gamma + \sum_{j=1}^D
    |\dual{\xxi^j}{(\tfrac12-\dlp)\ww_h+\slp\cchi_h}_\Gamma|^2 \Big) \\
    &= \min\{\c{monA},1\}\,\enorm{(\ww_h,\cchi_h)}^2.
  \end{align*}
  Finally, the norm equivalence of Proposition~\ref{prop:richenough}
  yields strong monotonicity, where
  $\c{mon}=\min\{\c{monA},1\}\c{norm}^{-1}>0$ depends only on $\mon$, $\Omega$,
  $\lext$, $\mext$ and $\YY_0$.
  
  \textbf{Step~3 }(\textit{Unique solvability and C\'ea lemma}).
  The main theorem on strongly monotone operators,
  see e.g.~\cite[Section~25]{zeidler}, states that the operator formulation 
  ~\eqref{eq:sym:operatoreq} and thus
  the Galerkin formulation~\eqref{eq:sym:weakform:disc} admits a unique 
  solution $(\uu_h,\pphi_h)\in\HH_h$.
  For $\HH_h=\HH$, we see that also the symmetric formulation~\eqref{eq:sym:weakform4},
  admits a unique solution $(\uu,\pphi)\in\HH$.
  Finally, standard theory~\cite[Section~25]{zeidler} also proves the validity of C\'ea's lemma~\eqref{eq:sym:cea}, 
  where $\c{cea}=\c{lip}/\c{mon}>0$ depends only on $\Omega$, $\mon$, $\lext$,
  $\mext$ and $\YY_0$.
\qed

\begin{remark}
  Our analysis unveils that~\eqref{eq2:sym:weakform} tackles the rigid body
  motions in the interior domain. 
  We have seen in~\eqref{eq:sym:notmonotone} that this information is lost when
  trying to prove ellipticity of $b(\cdot,\cdot)$, but can be reconstructed by
  adding appropriate terms to $b(\cdot,\cdot)$.
  We stress that the radiation condition~\eqref{eq:modelproblem:radcond}
  fixes the rigid body motion in the exterior $\Omegaext$, see
  also Section~\ref{sec:modelproblem:tp}.
  Since the interior and exterior solution are coupled via
  equation~\eqref{eq2:sym:weakform} this information is transferred
  by~\eqref{eq2:sym:weakform} from the exterior to the interior. 
  Thus, adding terms to $b(\cdot,\cdot)$ that satisfy~\eqref{eq2:sym:weakform}
  for fixed test-functions seems to be a natural approach.
\end{remark}

%-----------------------------------------------------------------
\subsection{Proof of Theorem~\ref{thm:assumption}}\label{sec:proof:assumption}
%-----------------------------------------------------------------
Let $\rr^1,\dots,\rr^D$ be a basis of the rigid body motions $\RR_d$
and let $\Pi_0:\LL^2(\Gamma)\to\PP^0(\EE_h^\Gamma)$ denote the $\LL^2$-projection.
We shall use the observation from the proof of Proposition~\ref{prop:richenough} that
assumption~\eqref{eq:assonYYh} is equivalent to the fact that $\Pi_0(\rr^j)$, for $j=1,\dots,D$
are linearly independent.

\begin{proof}[Proof of Theorem~\ref{thm:assumption} for $d=2$]
 Let
  \begin{align*}
    \rr^1:=\vector10{}, \quad \rr^2:=\vector01{}, \quad
    \rr^3:=\vector{-\xx_2}{\xx_1}{}
  \end{align*}
  denote the canonical basis of $\RR_2$, and let $\alpha_1,\alpha_2,\alpha_3\in\R$ fulfill
  \begin{align}\label{eq:lindep2d}
    \alpha_1 \Pi_0(\rr^1) + \alpha_2 \Pi_0(\rr^2) +\alpha_3 \Pi_0(\rr^3) =
    \bignull.
  \end{align}
  We stress that $\Pi_0(\rr^1) = \rr^1$ and $\Pi_0(\rr^2)=\rr^2$. For
  $E\in\EE_h^\Gamma$, we get
  \begin{align*} 
    \Pi_0(\rr^3)|_E = \frac1{|E|} \vector{-\int_E \xx_2 \,d\Gamma_{\xx}}{\int_E
    \xx_1 \,d\Gamma_{\xx}}{} =
    \vector{-\cm_2^E}{\cm_1^E}{},
  \end{align*}
  where $\cm^E = (\cm_1^E,\cm_2^E)^T$ denotes the midpoint of a boundary element $E$.
  Therefore,~\eqref{eq:lindep2d} can be written as
  \begin{align}\label{eq:lindep2d1}
    \vector{\alpha_1}{\alpha_2}{} + \alpha_3\vector{-\cm_2^E}{\cm_1^E}{} = \bignull
    \quad\text{for all }E\in\EE_h^\Gamma.
   \end{align}
  Altogether, we thus obtain $\alpha_3\cm^E = \alpha_3\cm^{E'}$ for all
  $E,E'\in\EE_h^\Gamma$, which can only hold for $\alpha_3=0$.
  This implies $\alpha_1\rr^1+\alpha_2\rr^2=0$ and hence
  $\alpha_1=0=\alpha_2$. Therefore, $\Pi_0(\rr^j)$,
  $j=1,\dots,3=D$, are linearly independent which is equivalent to~\eqref{eq:assonYYh}.
\end{proof}

\begin{proof}[Proof of Theorem~\ref{thm:assumption} for $d=3$]
  Let
  \begin{align*}
    \rr^1:=\vector100, \rr^2:=\vector010, \rr^3:=\vector001,
    \rr^4:=\vector{-\xx_2}{\xx_1}0, \rr^5:= \vector0{-\xx_3}{\xx_2},
    \rr^6:=\vector{\xx_3}0{-\xx_1}
  \end{align*}
  denote the canonical basis of $\RR_3$. We stress that
  $\Pi_0(\rr^j)=\rr^j$ for $j=1,2,3$, and
  \begin{align*}
    \Pi_0(\rr^4)|_E = \vector{-\cm_2^E}{\cm_1^E}0, \quad
    \Pi_0(\rr^5)|_E = \vector0{-\cm_3^E}{\cm_2^E}, \quad
    \Pi_0(\rr^6)|_E = \vector{\cm_3^E}0{-\cm_1^E}
  \end{align*}
  for all faces $E\in\EE_h^\Gamma$, where $\cm^E =
  (\cm_1^E,\cm_2^E,\cm_3^E)^T\in\R^3$ denotes the center of mass of an element
  $E\in\EE_h^\Gamma$. The main ingredient for the proof is the geometric
  observation of Lemma~\ref{lemma:geom3d} from the Appendix:
  There are at least three elements $A,B,C\in\EE_h^\Gamma$ such that the corresponding
  centers of mass $\aa,\bb,\cc$ do not lie on one line. 
  Let $\alpha_1,\alpha_2,\alpha_3,\alpha_4,\alpha_5,\alpha_6
  \in\R$ fulfill
  \begin{align*}
    \alpha_1\Pi_0(\rr^1)+\alpha_2\Pi_0(\rr^2)+\alpha_3\Pi_0(\rr^3)+
    \alpha_4\Pi_0(\rr^4)+\alpha_5\Pi_0(\rr^5)+\alpha_6\Pi_0(\rr^6) = \bignull,
  \end{align*}
  which is equivalent to
  \begin{align}\label{eq:lindep3d}
    \vector{\alpha_1}{\alpha_2}{\alpha_3} + 
    \begin{pmatrix} -\cm_2^E & 0 & \cm_3^E \\
                    \cm_1^E & -\cm_3^E & 0 \\
		    0 & \cm_2^E & -\cm_1^E \end{pmatrix}
    \vector{\alpha_4}{\alpha_5}{\alpha_6} = \vector000
  \end{align}
  for all $E\in\EE_h^\Gamma$.
  We take the third equation of~\eqref{eq:lindep3d} for the three elements $A,B,C$
  corresponding to $\aa,\bb,\cc$ and get
  \begin{align*}
    \begin{pmatrix}
      1 & \aa_2 & -\aa_1 \\
      1 & \bb_2 & -\bb_1 \\
      1 & \cc_2 & -\cc_1
    \end{pmatrix}
    \vector{\alpha_3}{\alpha_5}{\alpha_6} = \bignull,
  \end{align*}
  which is only satisfied if $\alpha_3=\alpha_5=\alpha_6=0$ or if 
  the vectors
  \begin{align}\label{eq:lindep3d2}
    \vector111, \vector{\aa_2}{\bb_2}{\cc_2}, \vector{-\aa_1}{-\bb_1}{-\cc_1}
    \quad\text{are linearly dependent.}
  \end{align}
  
  \textbf{Case 1 } ($\alpha_3=\alpha_5=\alpha_6=0$).
  We insert
  $\alpha_3,\alpha_5,\alpha_6$ into the first two equations
  of~\eqref{eq:lindep3d} and infer for all $E\in\EE_h^\Gamma$
  \begin{align*}
    \vector{\alpha_1}{\alpha_2}{} + \alpha_4 \vector{-\cm_2^E}{\cm_1^E}{} =
    \bignull \quad\text{or equivalently }
  \end{align*}
  \begin{align*}
    \alpha_4 \vector{\cm_1^E}{\cm_2^E}{} =
    \alpha_4\vector{\cm_1^{E'}}{\cm_2^{E'}}{} \quad\text{for all
    }E,E'\in\EE_h^\Gamma,
  \end{align*}
  which can only hold if $\alpha_4=0$. Otherwise, all centers of mass would
  lie on a straight line which would contradict the choice of $A,B,C$. Therefore,
  we first get $\alpha_3=\alpha_4=\alpha_5=\alpha_6=0$, and this also implies 
  $\alpha_1=\alpha_2=0$.

  \textbf{Case 2 } (\eqref{eq:lindep3d2} holds). 
  There exist constants $\gamma',\delta' \in\R$ such that
  \begin{align}\label{eq:lindep3dcase2}
    \vector{\aa_1}{\bb_1}{\cc_1} = \gamma'  \vector{\aa_2}{\bb_2}{\cc_2} +
    \delta' \vector111.
  \end{align}
  Next, we take the first equation of~\eqref{eq:lindep3d} for the three elements
  corresponding to $\aa,\bb,\cc$ and get
  \begin{align*}
    \begin{pmatrix}
      1 & -\aa_2 & \aa_3 \\
      1 & -\bb_2 & \bb_3 \\
      1 & -\cc_2 & \cc_3
    \end{pmatrix}
    \vector{\alpha_1}{\alpha_4}{\alpha_6} = \bignull,
  \end{align*}
  which is only fulfilled if $\alpha_1=\alpha_4=\alpha_6=0$ or if the vectors
  \begin{align}\label{eq:lindep3d3}
    \vector111, \vector{-\aa_2}{-\bb_2}{-\cc_2}, \vector{\aa_3}{\bb_3}{\cc_3}
    \quad\text{are linearly dependent.}
  \end{align}
  
  \textbf{Case 2a } ($\alpha_1=\alpha_4=\alpha_6=0$). We insert
  $\alpha_1,\alpha_4,\alpha_6$ into equation two and three
  in~\eqref{eq:lindep3d} and infer for all $E\in\EE_h^\Gamma$
  \begin{align*}
    \vector{\alpha_2}{\alpha_3}{} + \alpha_5 \vector{-\cm_3^E}{\cm_2^E}{} =
    \bignull \quad\text{or equivalently }
  \end{align*}
  \begin{align*}
    \alpha_5 \vector{\cm_2^E}{\cm_3^E}{} =
    \alpha_5\vector{\cm_2^{E'}}{\cm_3^{E'}}{} \quad\text{for all
    }E,E'\in\EE_h^\Gamma,
  \end{align*}
  which implies $\alpha_5=0$ as in Case 1.
  Arguing as above, we first get $\alpha_1=\alpha_4=\alpha_5=\alpha_6=0$ and finally
  also $\alpha_2=\alpha_3=0$.
  
  \textbf{Case 2b } (\eqref{eq:lindep3d3} holds). 
  There exist constants $\lambda,\mu\in\R$ such that
  \begin{align*}
    \vector{\aa_2}{\bb_2}{\cc_2} = \lambda  \vector{\aa_3}{\bb_3}{\cc_3} +
    \mu \vector111.
  \end{align*}
  Together with~\eqref{eq:lindep3dcase2}, we get
  \begin{align*}
    \aa = \aa_3 \vector\gamma\lambda1 + \vector\delta\mu0, \quad
    \bb = \bb_3 \vector\gamma\lambda1 + \vector\delta\mu0, \quad
    \cc = \cc_3 \vector\gamma\lambda1
    + \vector\delta\mu0,
  \end{align*}
  where $\gamma = \gamma'\lambda$ and $\delta = \gamma'\mu+\delta'$. This means that
  $\aa,\bb,\cc$ lie on one line which contradicts our choice of
  the elements $A,B,C$. In particular, case 2b cannot occur.

  Altogether, we have shown
  $\alpha_1=\alpha_2=\alpha_3=\alpha_4=\alpha_5=\alpha_6=0$
  in~\eqref{eq:lindep3d}, and therefore the orthogonal projections
  $\Pi_0(\rr^j)$, $j=1,\dots,6=D$, are linearly independent. Since this is 
  equivalent to~\eqref{eq:assonYYh}, we conclude the proof.
\end{proof}

%%%%%%%%%%%%%%%%%%%%%%%%%%%%%%%%%%%%%%%%%%%%%%%%%%%%%%%
% JN Coupling
%%%%%%%%%%%%%%%%%%%%%%%%%%%%%%%%%%%%%%%%%%%%%%%%%%%%%%
\section{Johnson-N\'ed\'elec coupling}\label{sec:jn}
\noindent
This section deals with the Johnson-N\'ed\'elec coupling, see
e.g.~\cite{johned,zkb79} for linear Laplace problems and~\cite{ghs09,s3} for
linear elasticity problems.
In contrast to~\cite{ghs09} resp. \cite{s3} we avoid the use of interior
Dirichlet boundaries resp. an explicit stabilization of the coupling equations.
The derivation of the variational formulation~\eqref{eq:jn:weakform} of the
Johnson-N\'ed\'elec coupling and the proof of equivalence to the model
problem~\eqref{eq:modelproblem} are done as for the Laplace problem, see
e.g.~\cite{affkmp,gh95} for the derivation.

\subsection{Variational formulation}\label{sec:jn:weakform}
The Johnson-N\'ed\'elec coupling reads as follows: Find $(\uu,\pphi)\in\HH =
\H^1(\Omega)\times\H^{-1/2}(\Gamma)$ such that
\begin{subequations}\label{eq:jn:weakform}
\begin{align}
  \dual{\mon\strain(\uu)}{\strain(\vv)}_\Omega - \dual{\pphi}{\vv}_\Gamma
  &=\dual{\ff}{\vv}_\Omega + \dual{\pphi_0}\vv_\Gamma \\
  \dual{\ppsi}{(\tfrac12-\dlp)\uu+\slp\pphi}_\Gamma &=
  \dual\ppsi{(\tfrac12-\dlp)\uu_0}_\Gamma
\end{align}
holds for all $(\vv,\ppsi)\in\HH$.
\end{subequations}
Note that the second equation of the Johnson-N\'ed\'elec equations
\eqref{eq:jn:weakform} is the same as for the symmetric
coupling~\eqref{eq:sym:weakform}.
We define a mapping $b:\HH\times\HH\to\R$ and a
continuous linear functional $\FF\in\HH^*$ by
\begin{align}\label{eq:jn:defbb}
  b( (\uu,\pphi),(\vv,\ppsi)) := 
  \dual{\mon\strain(\uu)}{\strain(\vv)}_\Omega - \dual{\pphi}{\vv}_\Gamma
  +\dual{\ppsi}{(\tfrac12-\dlp)\uu+\slp\pphi}_\Gamma 
\end{align}
as well as
\begin{align}\label{eq:jn:defFF}
  \FF(\vv,\ppsi) := \dual{\ff}{\vv}_\Omega + \dual{\pphi_0}\vv_\Gamma +
  \dual\ppsi{(\tfrac12-\dlp)\uu_0}_\Gamma
\end{align}
for all $(\uu,\pphi),(\vv,\ppsi)\in\HH$. 
Problem~\eqref{eq:jn:weakform} can equivalently be stated as follows: Find
$(\uu,\pphi)\in\HH$ such that
\begin{align}\label{eq:jn:weakform2}
  b( (\uu,\pphi),(\vv,\ppsi)) = \FF(\vv,\ppsi) \quad\text{holds for all } (\vv,\ppsi)\in\HH. 
\end{align}
We infer from~\eqref{eq:jn:defbb} that 
\begin{align} 
  b( (\rr,\bignull),(\rr,\bignull)) = 0 \quad\text{for }\rr\in\RR_d.
\end{align}
Therefore, the mapping
$b(\cdot,\cdot)$ cannot be elliptic, and we proceed as in
Section~\ref{sec:sym} to prove well-posedness of~\eqref{eq:jn:weakform} and its
Galerkin discretization.

\subsection{Main result}
According to~\cite{sw}, there exists a constant $1/2\leq c_\dlp<1$ such that
\begin{align}\label{eq:contractionDlp}
  \norm{(\tfrac12+\dlp)\vv}{\slp^{-1}} \leq c_\dlp \norm{\vv}{\slp^{-1}}
  \quad\text{for all } \vv\in\H^{1/2}(\Gamma),
\end{align}
where $\norm\vv{\slp^{-1}}^2 = \dual{\slp^{-1}\vv}\vv$ denotes an equivalent
norm on $\H^{1/2}(\Gamma)$ induced by the inverse of the simple-layer potential.
The following theorem is the main result of this section.
\begin{theorem}\label{thm:jn:solvability}
  Let $c_\dlp<1$ denote the contraction constant~\eqref{eq:contractionDlp} of the double-layer
  potential and assume that $2\c{monA} > c_\dlp (3\lext+2\mext)$. Then, the assertions of
  Theorem~\ref{thm:solvability} hold for the Johnson-N\'ed\'elec coupling
  accordingly.
\end{theorem}

\subsection{Auxiliary results}
We stress that the results of Section~\ref{sec:sym:stab}--\ref{sec:sym:norm}
also apply to the Johnson-N\'ed\'elec coupling without further modifications.
Additionally, the proof needs some properties of the boundary integral operators and some
results from the works~\cite{os,s3}, which are stated in the following.
First, we introduce the interior Steklov-Poincar\'e operator $\stp:
\H^{1/2}(\Gamma)\to\H^{-1/2}(\Gamma)$ defined by
\begin{align*}
  \stp := \slp^{-1}(\tfrac12+\dlp),
\end{align*}
see e.g.~\cite{hw}.
Note that $\slp$ and $\dlp$ are still defined with respect to the exterior Lam\'e
constants $\lext,\mext$.
We use the estimate
\begin{align*}
  \norm{(\tfrac12+\dlp)\ww}{\slp^{-1}}^2 \leq c_\dlp \dual{\stp\ww}\ww_\Gamma
  \quad\text{for all }\ww\in\H^{1/2}(\Gamma)
\end{align*}
from~\cite{os,s3}, which involves the contraction
constant~\eqref{eq:contractionDlp} of the double-layer potential $\dlp$. 
The last estimate yields
\begin{align}\label{eq:contraction}
  \dual\cchi{(\tfrac12+\dlp)\ww}_\Gamma \leq
  \norm{(\tfrac12+\dlp)\ww}{\slp^{-1}} \norm\cchi\slp \leq
  \sqrt{c_\dlp\dual{\stp\ww}\ww_\Gamma} \norm\cchi\slp\quad\text{for all }(\ww,\cchi)\in\HH.
\end{align}
For $\ww\in\H^1(\Omega)$ we next introduce the splitting
\begin{align}\label{eq:splitting}
\ww^0 := \ww - \wwd, 
\end{align}
where $\wwd\in\H^1(\Omega)$ is the unique weak solution of
\begin{align*}
  \div\,\stress^{\rm ext}(\wwd) &= 0 \quad\text{in }\Omega, \\
  \wwd &= \ww \quad\text{on }\Gamma.
\end{align*}
Then, there holds $\ww^0|_\Gamma = \bignull$ as well as the orthogonality relation
$\dual{\stress^{\rm ext}(\wwd)}{\strain(\ww^0)}_\Omega = 0 = \dual{\stress^{\rm
ext}(\ww^0)}{\strain(\wwd)}_\Omega$. 
Consequently, we see
\begin{align}\label{eq:splitid}
  \dual{\stress^{\rm ext}(\ww)}{\strain(\ww)}_\Omega = \dual{\stress^{\rm
  ext}(\wwd)}{\strain(\wwd)}_\Omega + \dual{\stress^{\rm
  ext}(\ww^0)}{\strain(\ww^0)}_\Omega.
\end{align}
Moreover, $\wwd$ fulfills $\gamma_1^{\rm int}\wwd = \stp\wwd$.
Together with Betti's first formula~\eqref{eq:betti}, we infer
\begin{align}\label{eq:stpid}
  \dual{\stress^{\rm ext}(\wwd)}{\strain(\wwd)}_\Omega = \dual{\gamma_1^{\rm int}
\wwd}{\wwd}_\Gamma = \dual{\stp\wwd}{\wwd}_\Gamma.
\end{align}

\subsection{Proof of Theorem~\ref{thm:jn:solvability}}
Note that Proposition~\ref{prop:equivalence} holds true with $b(\cdot,\cdot)$
resp.\ $F(\cdot)$
replaced by definition~\eqref{eq:jn:defbb} resp.~\eqref{eq:jn:defFF}.
We define the nonlinear operator $\widetilde\BB : \HH\to\HH^*$ by
\begin{align*}
  \dual{\widetilde\BB(\uu_h,\pphi_h)}{(\cdot,\cdot)} := \widetilde 
  b((\uu_h,\pphi_h),(\cdot,\cdot)).
\end{align*}

\textbf{Step~1 }(\textit{Lipschitz continuity of $\widetilde\BB$}).
Arguing as in~\eqref{eq:sym:lipBB} in the proof of Theorem~\ref{thm:solvability}, we
prove Lipschitz continuity of $\widetilde\BB$, where the Lipschitz constant
$\c{lip}>0$ depends only on $\mon$, $\lext$, $\mext$, and $\Omega$. 

\textbf{Step~2 }(\textit{Strong monotonicity of~$\widetilde\BB$}).
We have to prove that, for all $(\uu_h,\pphi_h),(\vv_h,\ppsi_h)\in\HH$,
\begin{align}
  \dual{\widetilde\BB(\uu_h,\pphi_h)-\widetilde\BB(\vv_h,\ppsi_h)}{(\uu_h-\vv_h,\pphi_h-\ppsi_h)}
  \geq \c{mon} \norm{(\uu_h-\vv_h,\pphi_h-\ppsi_h)}{\HH}^2.
\end{align}
To abbreviate notation, let $(\ww_h,\cchi_h) :=
(\uu_h-\vv_h,\pphi_h-\ppsi_h)$. By use of monotonicity~\eqref{eq:monA} of
$\mon$, we see
\begin{align*}  
  &\dual{\widetilde\BB(\uu_h,\pphi_h) -
  \widetilde\BB(\vv_h,\ppsi_h)}{\ww_h}_\Gamma \\
  &\quad = \dual{\mon\uu_h-\mon\vv_h}{\ww_h}_\Omega -
  \dual{\cchi_h}{\ww_h}_\Gamma + \dual{\cchi_h}{(\tfrac12-\dlp)\ww_h +
  \slp\cchi_h}_\Gamma \\
  &\qquad + \sum_{j=1}^D
  |\dual{\xxi^j}{(\tfrac12-\dlp)\ww_h+\slp\cchi_h}_\Gamma|^2 \\ 
  &\quad \geq \c{monA} \norm{\strain(\ww_h)}{\LL^2(\Omega)}^2 -
  \dual{\cchi_h}{(\tfrac12+\dlp)\ww_h}_\Gamma +
  \dual{\cchi_h}{\slp\cchi_h}_\Gamma  \\
  &\qquad + \sum_{j=1}^D |\dual{\xxi^j}{(\tfrac12-\dlp)\ww_h+\slp\cchi_h}_\Gamma|^2 \\
  &\quad =: I_1 -I_2 + I_3 + I_4.
\end{align*}
Next, we use the splitting~\eqref{eq:splitting} for $\ww_h = \ww^0 + \wwd$. 
Together with~\eqref{eq:splitid} and~\eqref{eq:stresscont}, where
$\c{stresscont} = 6\lext+4\mext$, we get 
\begin{align*} 
  I_1 \geq \frac{\c{monA}}{\c{stresscont}} \dual{\stress^{\rm
  ext}(\ww_h)}{\strain(\ww_h)}_\Omega &= \frac{\c{monA}}{\c{stresscont}}
  \dual{\stress^{\rm ext}(\ww^0)}{\strain(\ww^0)}_\Omega +
  \frac{\c{monA}}{\c{stresscont}}
  \dual{\stress^{\rm ext}(\wwd)}{\strain(\wwd)}_\Omega \\ &=: I_{11}+I_{12}.
\end{align*}
Estimate~\eqref{eq:contraction} and Young's inequality yield for $\delta>0$
\begin{align*} 
  I_2 = \dual{\cchi_h}{(\tfrac12+\dlp)\wwd}_\Gamma
  \leq \sqrt{c_\dlp\dual{\stp\wwd}\wwd_\Gamma} \norm{\cchi_h}\slp  
  \leq \frac\delta{2} c_\dlp\dual{\stp\wwd}\wwd_\Gamma 
  + \frac{\delta^{-1}}2 \norm{\cchi_h}\slp^2.
\end{align*}
With the last inequality and~\eqref{eq:stpid}, we get
\begin{align*}
  I_2 \leq \frac\delta{2} c_\dlp \dual{\stress^{\rm
  ext}(\wwd)}{\strain(\wwd)}_\Omega 
  + \frac{\delta^{-1}}2 \norm{\cchi_h}\slp^2.
\end{align*}
Now, we can further estimate the terms $I_{1}-I_2+I_3$ by
\begin{align*}
  I_1-I_2+I_3 &\geq I_{11} +
  \Big(\frac{\c{monA}}{\c{stresscont}}-\frac\delta{2}c_\dlp\Big)
  \dual{\stress^{\rm ext}(\wwd)}{\strain(\wwd)}_\Omega +
  \Big(1-\frac{\delta^{-1}}2\Big)\norm{\cchi_h}\slp^2 \\
  &\geq \Big(\frac{\c{monA}}{\c{stresscont}}-\frac\delta{2}c_\dlp\Big)
  \dual{\stress^{\rm ext}(\ww_h)}{\strain(\ww_h)}_\Omega +
  \Big(1-\frac{\delta^{-1}}2\Big) \dual{\cchi_h}{\slp\cchi_h}_\Gamma,
\end{align*}
where we used~\eqref{eq:splitid} again.
The assumption $2\c{monA}>c_\dlp(3\lext+2\mext)$ is equivalent to
$\c{monA} / \c{stresscont} > c_\dlp /4$ with $\c{stresscont} = 6\lext +4\mext$.
Therefore, there exists $\delta>0$
such that $C:=\min\{\c{monA}/\c{stresscont}-c_\dlp\delta/2,1-\delta^{-1}/2\}>0$.
Altogether, we infer with~\eqref{eq:splitid} and~\eqref{eq:stressell}
\begin{align*} 
\begin{split}
  I_1-I_2+I_3+I_4 &\geq C \Big(\dual{\stress^{\rm
  ext}(\ww_h)}{\strain(\ww_h)}_\Omega +
  \dual{\cchi_h}{\slp\cchi_h}_\Gamma + \sum_{j=1}^D
  |\dual{\xxi^j}{(\tfrac12-\dlp)\ww_h+\slp\cchi_h}_\Gamma|^2 \Big)\\
  &\geq \widetilde C  \Big(\norm{\strain(\ww_h)}{\LL^2(\Omega)}^2 +
  \dual{\cchi_h}{\slp\cchi_h}_\Gamma + \sum_{j=1}^D
  |\dual{\xxi^j}{(\tfrac12-\dlp)\ww_h+\slp\cchi_h}_\Gamma|^2 \Big) \\
  &= \widetilde C \enorm{(\ww_h,\cchi_h)}^2 \geq \widetilde C\c{norm}^{-1} \norm{(\ww_h,\cchi_h)}\HH^2,
\end{split}
\end{align*}
where $\widetilde C = C \min\{1,\c{stressell}\}$. The constant 
$\c{mon}:=\widetilde C\c{norm}^{-1}>0$ depends only on
$\Omega,\mon,\YY_0$, and on the Lam\'e constants $\lext,\mext$.

\textbf{Step~3 }(\textit{Unique solvability and C\'ea lemma}).
This step is essentially the same as Step~$3$ in the proof of
Theorem~\ref{thm:solvability}.
We thus omit the details.

\begin{remark}
  (i)
  In the linear case $\mon=\stress^{\rm int}$, we may also use an estimate
  from~\cite{s3} in Step~2 of the proof of Theorem~\ref{thm:jn:solvability} and replace the assumption
  $2\c{monA} > c_\dlp (3\lext+2\mext)$ from Theorem~\ref{thm:jn:solvability} with
  \begin{align*}
    \eta := \min\{\lambda^{\rm int}/\lext, \mu^{\rm int}/\mext\}>
    \frac{c_\dlp}4.
  \end{align*}
  \\
  (ii) The assumption $2\c{monA} > c_\dlp (3\lext+2\mext)$, is an assumption on
  the monotonicity constant $\c{monA}$ and the Lam\'e constants $\lext,\mext$ in
  the exterior domain.
  As we have seen for the symmetric coupling the assumption $\c{monA}>0$ suffices to prove
  unique solvability.
  Since the Johnson-N\'ed\'elec coupling is equivalent to the model problem, we
  stress that at least the continuous formulation of the Johnson-N\'ed\'elec
  coupling equations is uniquely solvable.
  In~\cite{os}, Of and Steinbach have shown that the Johnson-N\'ed\'elec coupling
  equations may become indefinite (and hence non-elliptic) for special choices of the model parameters.
  However, the numerical experiments from~\cite{affkmp} show at least
  numerically that the Laplace transmission problem also allows for unique Galerkin
  solutions in the indefinite regime.
  \\
  (iii) Assume a nonlinear Hencky-Von Mises stress-strain relation, i.e. the
  operator from~\eqref{eq:hencky}, with $\widetilde\mu(\cdot) \geq \alpha>0$
  and $\widetilde\mu(\cdot) \leq Kd/2 -\beta$ for some $\alpha,\beta>0$. 
  Then we may replace the assumption $2\c{monA} > c_\dlp (3\lext+2\mext)$ from
  Theorem~\ref{thm:jn:solvability} with
  \begin{align*}
    \eta > \frac{c_\dlp}4,
  \end{align*}
  where $\eta:= \min\{ \inf_{x\in\R_+}\{(K-2/d\widetilde\mu(x)\} / \lext,
  \inf_{x\in\R_+}\{\widetilde\mu(x)\} / \mext\}$.
\end{remark}

%%%%%%%%%%%%%%%%%%%%%%%%%%%%%%%%%%%%%%%%%%%%%%%%%%%%%%%%%%%%%%%%%%%%%%%%%%%%%%%%
% BMC coupling
%%%%%%%%%%%%%%%%%%%%%%%%%%%%%%%%%%%%%%%%%%%%%%%%%%%%%%%%%%%%%%%%%%%%%%%%%%%%%%%%
\section{Bielak-MacCamy coupling}\label{sec:bmc}
\noindent
In this section we investigate the non-symmetric Bielak-MacCamy one-equation
coupling, see e.g.~\cite{affkmp,bmc,coers} for the Laplace problem.
The derivation of the variational formulation~\eqref{eq:bmc:weakform} as well as
the proof of equivalence to the model problem~\eqref{eq:modelproblem}
essentially follow as for the Johnson-N\'ed\'elec coupling resp.
symmetric coupling, cf.\ e.g.~\cite{affkmp,cfs,gh95}.

\subsection{Variational formulation}
The variational formulation of the Bielak-MacCamy coupling reads as follows:
Find $(\uu,\pphi)\in\HH = \H^1(\Omega)\times\H^{-1/2}(\Gamma)$ such that
\begin{subequations}\label{eq:bmc:weakform}
\begin{align}\label{eq:bmc:weakforma}
  \dual{\mon\strain(\uu)}{\strain(\vv)}_\Omega + \dual{(\tfrac12-\dlp')\pphi}{\vv}_\Gamma
  &=\dual{\ff}{\vv}_\Omega + \dual{\pphi_0}\vv_\Gamma \\
  \dual{\ppsi}{\slp\pphi-\uu}_\Gamma &=
  -\dual\ppsi{\uu_0}_\Gamma
  \label{eq:bmc:weakformb}
\end{align}
holds for all $(\vv,\ppsi)\in\HH$.
\end{subequations}
We sum up the left-hand side and the right-hand side of~\eqref{eq:bmc:weakform}
and define the mapping $b:\HH\times\HH\to\R$ as well as the linear functional
$F\in\HH^*$ by
\begin{align}\label{eq:bmc:defbb}
  b( (\uu,\pphi),(\vv,\ppsi)) := 
  \dual{\mon\strain(\uu)}{\strain(\vv)}_\Omega + \dual{(\tfrac12-\dlp')\pphi}{\vv}_\Gamma
  +\dual{\ppsi}{\slp\pphi-\uu}_\Gamma 
\end{align}
as well as
\begin{align}\label{eq:bmc:defFF}
  \FF(\vv,\ppsi) := \dual{\ff}{\vv}_\Omega + \dual{\pphi_0}\vv_\Gamma - 
  \dual\ppsi{\uu_0}_\Gamma
\end{align}
for all $(\uu,\pphi),(\vv,\ppsi)\in\HH$. 
Then, problem~\eqref{eq:bmc:weakform} can equivalently be stated as follows: Find
$(\uu,\pphi)\in\HH$ such that
\begin{align}\label{eq:bmc:weakform2}
  b( (\uu,\pphi),(\vv,\ppsi)) = \FF(\vv,\ppsi) \quad\text{holds for all } (\vv,\ppsi)\in\HH. 
\end{align}
As for the other coupling formulations $b(\cdot,\cdot)$ is not uniformly elliptic, and unique solvability cannot be
shown directly. We follow the ideas of Section~\ref{sec:sym} resp.
Section~\ref{sec:jn} to overcome these difficulties.
Moreover, with $b_{\rm JN}(\cdot,\cdot)$ denoting the mapping defined
in~\eqref{eq:jn:defbb}, we stress that
\begin{align}\label{eq:bmceqjn} 
  b_{\rm JN}((\uu,\pphi),(\uu,\pphi)) = b((\uu,\pphi),(\uu,\pphi))
  \quad\text{for all } (\uu,\pphi)\in\HH.
\end{align}
Thus, there is a strong relation between the one-equation Bielak-MacCamy and
Johnson-N\'ed\'elec coupling.
In fact, for linear and symmetric $\mon : \LL^2(\Omega)\to\LL^2(\Omega)$, there
holds
\begin{align*}
  b_{\rm JN}((\uu,\pphi),(\vv,\ppsi)) = b((\vv,\ppsi),(\uu,\pphi))
  \quad\text{for all } (\uu,\pphi),(\vv,\ppsi)\in\HH.
\end{align*}

\subsection{Main result}
As in Section~\ref{sec:sym:stab}, we add appropriate terms to $b(\cdot,\cdot)$
to tackle the rigid body motions in the interior domain $\Omega$. 
In particular, we use~\eqref{eq:bmc:weakformb} to stabilize the linear form
$b(\cdot,\cdot)$.
We stress that Proposition~\ref{prop:equivalence} holds with $\widetilde
b(\cdot,\cdot)$ resp. $\widetilde F(\cdot)$ replaced by
\begin{align}
  \widetilde b( (\uu,\pphi),(\vv,\ppsi)) &:= b( (\uu,\pphi),(\vv,\ppsi)) +
  \sum_{j=1}^D \dual{\xxi^j}{\slp\pphi-\uu}_\Gamma
  \dual{\xxi}{\slp\vv-\ppsi}_\Gamma,  \\
  \widetilde F(\vv,\ppsi) &:= F(\vv,\ppsi) - \sum_{j=1}^D \dual{\xxi^j}{\uu_0}_\Gamma
  \dual{\xxi}{\slp\vv-\ppsi}_\Gamma 
\end{align}
for all $(\uu,\pphi),(\vv,\ppsi)\in\HH$.
Furthermore, the assertions of Proposition~\ref{prop:richenough} also hold true
if~\eqref{eq:richenough:gj} is replaced by
\begin{align}
  g_j(\uu,\pphi) := \dual\xxi{\slp\pphi - \uu}_\Gamma \quad\text{for }
  (\uu,\pphi)\in\HH.
\end{align}
With these observations, Theorem~\ref{thm:jn:solvability} holds
true for the Bielak-MacCamy coupling. Details are left to the reader.

\begin{remark}
  Our techniques developed in Section~\ref{sec:sym}--\ref{sec:bmc} can also be
  used for the (quasi-) symmetric Bielak-MacCamy coupling
  schemes~\cite{bmc,ghs09} applied to nonlinear elasticity problems. 
\end{remark}

%%%%%%%%%%%%%%%%%%%%%%%%%%%%%%%%%%%%%%%%%%%%%%%%%%%%%%%%%%%%%%%%%%%%%
% APPENDIX
\appendix
\section{Elementary geometric observation}

\begin{lemma}\label{lemma:geom3d}
  Let $d=3$ and $\EE_h^\Gamma$ be a regular triangulation of the closed boundary
  $\Gamma=\partial\Omega$ into flat surface triangles.
  Then, there are at least three triangles $A,B,C\in\EE_h^\Gamma$ such that the
  centers of mass $\aa,\bb,\cc$ corresponding to these elements do not lie on
  one line, i.e. $\cc-\aa \notin \{ t(\bb-\aa) \,:\, t\in\R \}$.
\end{lemma}

\begin{proof}
  We argue by contradiction. Assume that all centers of mass lie on one
  line $\lineg$.
  Let $\xx\in\KK_h^\Gamma$ denote an arbitrary node of the triangulation
  $\EE_h^\Gamma$. Recall that $\Gamma = \partial \Omega$ is the closed boundary
  of the polyhedral Lipschitz domain $\Omega$. Therefore, there are $k\geq3$ triangles
  $D_1,\dots,D_k\in\EE_h^\Gamma$ such that $\xx$ is a corner of $D_j$ for
  $j=1,\dots,k$.
  \begin{figure}[!htb]
    \begin{center}
      \psfrag{xx}{\tiny $\xx$}
      \psfrag{D1}{\tiny $D_1$}
      \psfrag{D2}{\tiny $D_2$}
      \psfrag{D3}{\tiny $D_3$}
      \psfrag{D4}{\tiny $D_4$}
      \psfrag{D5}{\tiny $D_5$}
      \psfrag{D6}{\tiny $D_6$}

      \psfrag{x1}{\tiny $\xx^1$}
      \psfrag{x2}{\tiny $\xx^2$}
      \psfrag{x3}{\tiny $\xx^3$}
      \psfrag{x4}{\tiny $\xx^4$}
      \psfrag{x5}{\tiny $\xx^5$}
      \psfrag{x6}{\tiny $\xx^6$}

      \includegraphics[width=0.5\textwidth]{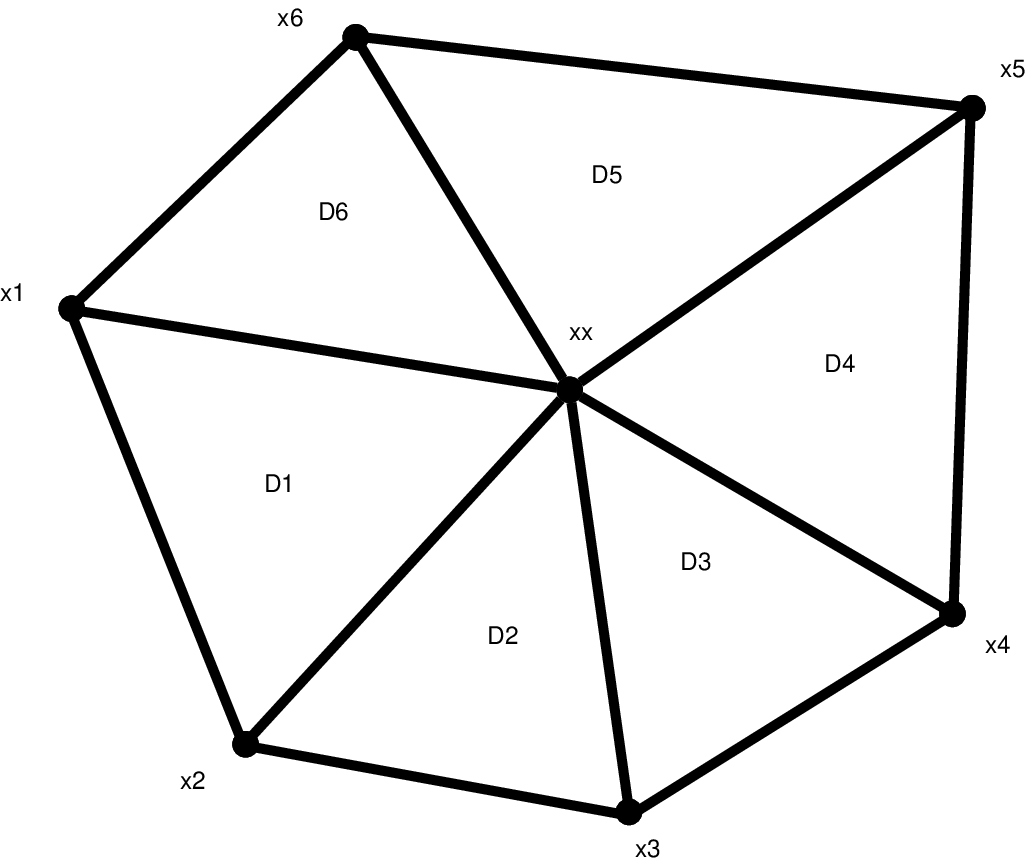}
    \end{center}
    \caption{For any node $\xx\in\KK_h^\Gamma$ in a regular triangulation
    $\EE_h^\Gamma$ of a closed boundary $\Gamma$, there
    exist $k\geq3$ different nodes $\xx^1,\dots,\xx^k\in\KK_h^\Gamma$ and
    triangles $D_1,\dots,D_k$ such that $D_j = \convhull{\xx,\xx^j,\xx^{j+1}}$ for
    all $j=1,\dots,k$. 
    Moreover, there holds $D_j\cap D_{j+1} = \convhull{\xx,\xx^{j+1}}$ for
    $j=1,\dots,k$ with $\xx^{k+1}=\xx^1$.
    Here, an example for $k=6$ is shown.}
    \label{fig:nodeWithTriangles}
  \end{figure}
  Moreover, let $\{\xx,\xx^1,\dots,\xx^k\}$ denote the set of all nodes of the
  triangles $D_1,\dots,D_k$. We stress that we can permute the indices of
  $\xx^1,\dots,\xx^k$ and the indices of $D_1,\dots,D_k$ such that $D_j =
  \convhull{\xx,\xx^j,\xx^{j+1}}$, where we define $\xx^{k+1}:=\xx^1$ and
  $\xx^{k+2}:=\xx^2$, see
  Figure~\ref{fig:nodeWithTriangles} for an illustration in the case $k=6$. 
  Let $\cm^1,\dots,\cm^k$ denote the centers of mass of the triangles
  $D_1,\dots,D_k$.
  Because of our assumption that all centers of mass lie on one line $\lineg$, we
  infer that
  \begin{align*}
    \cm^{j+1} - \cm^j = \frac{\xx^{j+2}+\xx^{j+1}+\xx}3
    -\frac{\xx^{j+1}+\xx^{j}+\xx}3 = \frac{\xx^{j+2}-\xx^j}3
  \end{align*}
  is proportional to the directional vector $\dd\neq\bignull$ of the line
  $\lineg$. Therefore,
  $\xx^{j+2}-\xx^j$ is also proportional to $\dd$.
  Moreover, we observe $t\dd = \cm^{j+3}-\cm^{j+2}+\cm^{j+1}-\cm^j =
  (\xx^{j+4}-\xx^j)/3$ for some $t\in\R$. By iterating this process, we
  get with appropriate $t_m, t_n \in \R$
  \begin{align*}
    t_m\dd &=\sum_{j=1}^{2m} (-1)^j \cm^j = (\xx^{2m+1}-\xx^1)/3 \quad\text{for
    all $m$ with } 2m\leq k, \\
    t_n\dd &= \sum_{j=2}^{2n-1} (-1)^{j+1} \cm^j = (\xx^{2n} - \xx^2)/3
    \quad\text{for all $n$ with } 2n-1\leq k.
  \end{align*}
  Altogether, we see that all nodes with even indices lie on one line $\lineh$, and
  all nodes with odd indices lie on one parallel line $\linef$, i.e.
  \begin{align*}
    \xx^{2j} \in \left\{ \xx^2 + t\dd \,:\, t\in\R \right\} &=: \lineh
    \quad\text{and} \\
    \xx^{2j-1} \in \left\{ \xx^1 + t\dd \,:\, t\in\R \right\} &=: \linef
  \end{align*}
  for all $j\in\N$ with $2j\leq k$ resp. $2j-1\leq k$.
  For the remainder of the proof, we distinguish whether $k$ is odd or
  even.

  \textbf{Case 1 } ($k$ is odd).
  The observations above show that $\xx^1,\xx^3,\dots,\xx^k \in \linef$ and
  $\xx^2-\xx^k = \xx^{k+2}-\xx^k = t\dd$ for some $t\in\R$. Then, $\xx^2\in
  \linef$
  and since $\linef$ and $\lineh$ are parallel, there holds $\lineh=\linef$, which means that all
  nodes $\xx^1,\dots,\xx^k$ lie on one line.
  This contradicts a regular triangulation.

  \textbf{Case 2 } ($k$ is even).
  If $\lineh=\linef$ we can argue as in Case~1.
  Otherwise $\lineh\neq \linef$, and we stress that all edges $\overline{\xx^j\xx^{j+1}}$
  of the triangles are connected, i.e.
  \begin{align*}
    \overline{\xx^j\xx^{j+1}} \cap \overline{\xx^n\xx^{n+1}} = 
    \begin{cases}
      \{\xx^j\} & \text{if } j=n+1 \\
      \{\xx^{j+1}\} & \text{if } j+1=n \\
      \emptyset & \text{otherwise}.
    \end{cases}
  \end{align*}
  Moreover, every edge $\overline{\xx^j\xx^{j+1}}$ connects the lines $\lineh$ and
  $\linef$.
  Thus, we can infer that there are two edges $\overline{\xx^m\xx^{m+1}}$,
  $\overline{\xx^n\xx^{n+1}}$ which intersect each other, i.e.
  \begin{align*}
    \overline{\xx^m\xx^{m+1}}\cap \overline{\xx^n\xx^{n+1}} = \{\yy\}
    \quad\text{with } \yy \neq \xx^n \text{ and } \yy\neq \xx^{n+1},
  \end{align*}
  see also Figure~\ref{fig:intersect} for an illustration.
  Altogether this contradicts a regular triangulation.
  \begin{figure}[!htb]
    \begin{center}
      \psfrag{h}{\tiny $\lineh$}
      \psfrag{f}{\tiny $\linef$}
      \psfrag{x2n}{\tiny $\xx^{2n}$}
      \psfrag{x2n1}{\tiny $\xx^{2n+1}$}
      \psfrag{x2m}{\tiny $\xx^{2m}$}
      \psfrag{x2m1}{\tiny $\xx^{2m+1}$}

      \includegraphics[width=0.5\textwidth]{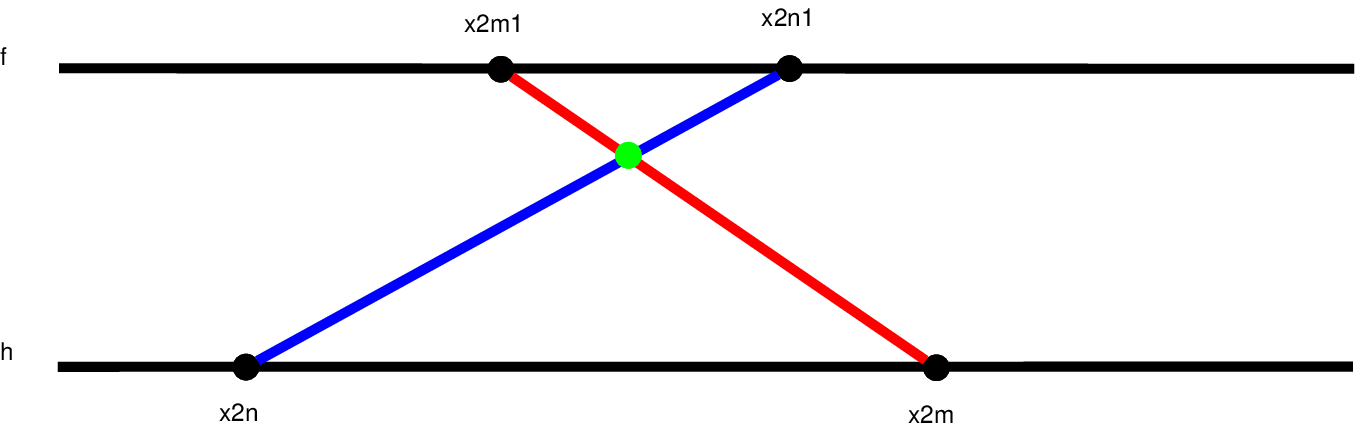}
    \end{center}
    \caption{Let $\xx\in\KK_h^\Gamma$ be an arbitrary node in the regular
    triangulation $\EE_h^\Gamma$ and let $\xx^1,\dots,\xx^k$ denote the
    neighboring nodes. Under the assumption that all centers of mass of the
    triangles in $\EE_h^\Gamma$ lie on a line $\lineg$, the proof of
    Lemma~\ref{lemma:geom3d} unveils that the nodes $\xx^{2j-1}$ lie on a line
    $\linef$ and the nodes $\xx^{2j}$ lie on a line $\lineh$, which is parallel to
    $\linef$.
    Since all $\xx^j$ are connected by the segments $\overline{\xx^j\xx^{j+1}}$,
    we can conclude that there are indices $n,m$ with $2n\leq k, 2m\leq
    k$ such that $\overline{\xx^{2n}\xx^{2n+1}}\cap
    \overline{\xx^{2m}\xx^{2m+1}} = \{\yy\}$ and $\yy\notin\KK_h^\Gamma$.
    }
    \label{fig:intersect}
  \end{figure}
\end{proof}

%%%%%%%%%%%%%%%%%%%%%%%%%%%%%%%%%%%%%%%%%%%%%%%%%%%%%%%%%%%%%%%%%%%%%
% Bibliography
\bibliographystyle{alpha}
\bibliography{literature}
%%%%%%%%%%%%%%%%%%%%%%%%%%%%%%%%%%%%%%%%%%%%%%%%%%%%%%%%%%%%%%%%%%%%%

\end{document}